\documentclass[11pt,oneside,english,shortlabels]{amsart}

\usepackage{babel}
\usepackage{amstext}
\usepackage{amssymb}

\usepackage{amsmath} 
\usepackage{latexsym}
\usepackage{graphicx} 

\usepackage{mathptmx}

\usepackage{graphicx} 

\usepackage{xcolor}

\makeatletter

\usepackage{amsthm}
\usepackage{enumitem}		

\usepackage{mathrsfs} 
\providecommand{\noopsort}[1]{}

\numberwithin{equation}{subsection}

\theoremstyle{definition} 

 \newtheorem{definition}{Definition}[section]
 \newtheorem{remark}[definition]{Remark}

\newtheorem{question}[definition]{Question}

\theoremstyle{plain}      

 \newtheorem{proposition}[definition]{Proposition}
 \newtheorem{theorem}[definition]{Theorem}
 \newtheorem{corollary}[definition]{Corollary}
 \newtheorem{lemma}[definition]{Lemma}

\makeatletter
\newcommand*{\house}[1]{
  \mathord{
    \mathpalette\@house{#1}
  }
}
\newcommand*{\@house}[2]{
  \dimen@=\fontdimen8 %
      \ifx#1\scriptscriptstyle\scriptscriptfont
      \else\ifx#1\scriptstyle\scriptfont
      \else\textfont\fi\fi
      3 %
  \sbox0{%
    $#1%
      \vrule width\dimen@\relax
      \overline{%
        \kern2\dimen@
        \begingroup 
          #2%
        \endgroup
        \kern2\dimen@
      }
      \vrule width\dimen@\relax
      \mathsurround=1.5\dimen@ 
    $
  }
  \ht0=\dimexpr\ht0-\dimen@\relax
  \dp0=\dimexpr\dp0+2\dimen@\relax
  \vbox{
    \kern\dimen@ 
    \copy0 
  }
}

\def\11{{\mathbf 1}}

\def\l{\lambda}
\def\m{\mu}
\def\a{\alpha}
\def\b{\beta}
\def\g{\gamma}
\def\d{\delta}
\def\e{\epsilon}

\def\O{\Omega}

\def\t{\theta}

\theoremstyle{remark}

\def \a{{\alpha}}
\def \b{{\beta}}
\def \D{{\Delta}}

\def \d{{\delta}}
\def \e{{\varepsilon}}

\def \g{{\gamma}}

\def \k{{\kappa}}
\def \l{{\lambda}}

\def \O{{\Omega}}
\def \p{{\varphi}}
\def \t{{\vartheta}}

\def \m{{\mu}}
\def \s{{\sigma}}

\def \E{{\bf E}\, }

\def \P{{\bf P}}

\def \qq{{\qquad}}
\def \R{{\bf R}}

\def \Z{{\bf Z}}

\def \dd{{\rm d}}

\def \noi{{\noindent}}

\def\E{{\mathbb E \,}}

\def\P{{\mathbb P}}

\def\R{{\mathbb R}}

\def\Z{{\mathbb Z}}

\numberwithin{equation}{subsection}

\author{Michel Weber$\mbox{}^{\dag}$}
\thanks{}
\address{$\mbox{}^{\dag}$
IRMA, 10  Rue du G\'en\'eral-Zimmer, 67084 Strasbourg,  Cedex, France. 
}
\email{michel.weber@math.unistra.fr}
\urladdr{}

\title[$\textit{Critical probabilistic and arithmetical characteristics  of the  Cram\'er  model for primes}$]
{Critical probabilistic characteristics of the  Cram\'er    model for primes
 and     arithmetical   properties}

\begin{document}


\begin{abstract}  
This work is a probabilistic study of the \lq primes\rq~ of the Cram\'er model, which  consists with     
sums $S_n =\sum_{i=3}^n \xi_i$, $n\ge 3$, where $\xi_i$ are     independent   random variables   such that $\P\{\xi_i= 1\}= 1-\P\{\xi_i= 1\}=1/{\log i}$, $i\ge3$.  
    We   prove 
    that there exists a set    of integers $\mathcal S $  of density 1  such that 
\begin{equation}\label{abs.3}   \liminf_{  \mathcal S\ni n\to\infty} (\log n)\P \{S_n\ \hbox{prime}  \}
\ge  \frac{1}{\sqrt{2\pi e}\, } ,
  \end{equation}
and that for $b>\frac12$,  the formula 
\begin{equation}\label{abs.4}  
 \P \{S_n\ \text{prime}\, \}
\, =\, \frac{ (1+ o( 1) )}{ \sqrt{2\pi  B_n  } } 
\int_{m_n-\sqrt{ 2bB_n\log n}}^{m_n+\sqrt{ 2bB_n\log n}}
  \,  e^{- \frac{(t- m_n)^2}{    2  B_n   } }\, \dd \pi(t),  \end{equation} 
in which        $m_n=\E S_n,B_n={\rm Var }\,S_n$,
  holds true  for all $n\in \mathcal S$, $n\to \infty$.
\vskip 1 pt
  Further we prove that for any $0<\eta<1$, and all $n$ large enough  and      $ \zeta_0\le  \zeta\le \exp\big\{ \frac{c\log n}{\log\log n}\big\}$,   letting $S'_n= \sum_{j= 8}^n \xi_j$, 
\begin{eqnarray*}  \P\big\{  S'_n\hbox{\ $\zeta$-quasiprime}\big\}   \,\ge  \, (1-\eta)   \frac{ e^{-\gamma} }{  \log
\zeta }, \end{eqnarray*}
according to Pintz's terminology, where  $c>0$  and   $\g$ is  Euler's constant.  We also  test which   infinite sequences of primes    are ultimately avoided  by   the  \lq primes\rq~ of the Cram\'er model, with probability 1.

\vskip 2 pt Moreover we  show that the Cram\'er model has incidences on     the Prime Number Theorem, since 
it  predicts that the error term is sensitive to   subsequences.
We 
obtain sharp results on   the length
  and  the number of occurences
  of  intervals $I$ such as  for some $z>0$,
\begin{equation}\label{abs.2} \sup_{n\in I} \frac{|S_n-m_n|}{ \sqrt{B_n}}\le z, 
\end{equation}
which are tied with   the spectrum of the Sturm-Liouville equation.

    \end{abstract}

\maketitle

\vspace{0.7cm}

\vspace{0.5cm}




\tableofcontents  

\section{Introduction.}\label{s1}
\numberwithin{equation}{section} \numberwithin{figure}{section}
 
 Let $\mathcal P=\{p_i,i\ge 1\}$ denote   the sequence of consecutive prime numbers. Cram\'er's probabilistic model basically consists with  a sequence of independent random variables $\xi_i$, defined for 
$i\ge 3$ by
\begin{equation}\label{C.xin}\P\{\xi_i= 1\}= \frac1{\log i}, \qq \quad \P\{\xi_i= 0\}= 1-\frac1{\log i}.
\end{equation}
  
 \vskip 3 pt This
  work 
  is   a probabilistic   study 
 of the \lq primes\rq\, of the Cram\'er model, most of the results   obtained have an easy arithmetical interpretation. 
 We show 
   that the  Cram\'er  \lq primes\rq\,   are 
    contained in the set of  \lq primes\rq\, of  the Bernoulli model. 
   This is applied
  to test which infinite sequences of primes   are with probability 1, ultimately avoided  by the \lq primes\rq\, of the model.
 We further thoroughly study the probability $ \P \{S_n\, \text{prime}\, \}$ and $\P \{  S_n\hbox{\ $\zeta$-quasiprime} \}$ (sections \ref{7}, \ref{5}).      
  \vskip 3 pt 
 We also  describe    new results of different type. Some preliminary facts,  at first   it is  easy to check that for this model, the standard limit theorems from probability theory are fulfilled:  the strong law of large numbers (SLLN), the  central limit theorem (CLT), the law of the iterated logarithm (LIL), the local limit theorem (LLT),   and also an invariance principle (IP) hold (Proposition \ref{cramer.IP}). This is true in a wider setting. These points are briefly detailed and completed in Appendix \ref{appendix-1}, which also contains  a sharp estimate of the characteristic function of $S_n =\sum_{i=3}^n \xi_i$  and the  value-distribution description of the divisors of $S_n$. 
Let $\mathcal C=\{ j\ge 3\!:\,\xi_j=1\}$. Note that obviously $S_x=\#\big\{ \nu\in \mathcal C \!:\nu\le x\big\}$  for all  reals $x\ge 3$. In particular, the LIL implies that 
\begin{equation}\label{cramer.pnt}
 \#\big\{ \nu\in \mathcal C : \nu\le x\big\} = \int_2^x \frac{\dd t}{\log t} + \mathcal O\big( \sqrt{x\log\log x}\big) ,
\end{equation}
with probability one. 

\vskip 3 pt 
Thus by \eqref{cramer.pnt},
\begin{equation}\label{cramer.pnt.1}
 \#\big\{ \nu\in \mathcal C : \nu\le x\big\}  = \int_2^x \frac{\dd t}{\log t} + \mathcal O_\e\big(  x^{\frac12 +\e}\big),
\end{equation}
with probability one.
The same result for the prime sequence $\mathcal P$ is equivalent to the Riemann Hypothesis (RH). 
\vskip 2 pt
The LIL 
is   a consequence of Kolmogorov's LIL, and yields the more precise result
\begin{equation}\label{cramer.pnt.LIL}
\limsup_{x\to \infty}
\frac{ \#\big\{ \nu\in \mathcal C : \nu\le x\big\} -\int_2^x \frac{\dd t}{\log t}}{ \sqrt{2 \,\big(\frac{ x}{\log x}\big)  \log\log x }} \,=\, 1,
\end{equation} 
with probability one.
\vskip 10 pt 
  We 
  question  the analogy made with \eqref{cramer.pnt.1} and prove   that   this model  possesses    finer tied properties, enlighting the above analogy  somehow differently.
 We notably prove that if $x$ runs along any increasing subsequence of integers $\mathcal N$, 
  \begin{equation}\label{cramer.pnt1a} 
\#\big\{ \nu\in \mathcal C : \nu\le x\big\} = \int_2^x \frac{\dd t}{\log t} + \mathcal O\big( \sqrt{x }\,\p_{\mathcal N}(x)\big),
\end{equation}
with probability one.  And we may have that $\p_{\mathcal N}(x)= o(\sqrt{\log\log x})$, in fact $\p_{\mathcal N}(x)$ can be as slow as desired, along  a  suitable subsequence $\mathcal N$.
 \vskip 3pt 

  \vskip 3pt 
Thus \eqref{cramer.pnt1a} implies that \lq the prime number theorem\rq\, in Cram\'er's model is sensitive to the subsequence on which $x$ is running, which seems not corroborated with any existing result concerning the counting function $\pi(x)$.  

\vskip 3pt 

  \vskip 3pt

       In the limiting  case $\p_{\mathcal N}(x)\!\equiv{\rm Const} $, we      study      the number of occurences  and the length
   of     intervals $I$   for which
\begin{equation}\label{cramer.small.ampl.} \sup_{n\in I} \frac{|S_n-\E S_n|}{ \sqrt{{\rm Var} (S_n)}}\le z,
\end{equation}
  $z$ being some positive real.   Such a property, namely the maximal duration of small amplitudes of $S_n$ around $\E S_n$,   is quite sensitive to the value taken by $z$, and   turns up to be tied with   the spectrum of the Sturm-Liouville equation. We obtain  in Theorems \ref{cramer.small.ampl.t1}, \ref{cramer.frequencies.t1} quite sharp results.  The proofs combine the IP with   small local oscillation   results   of  the Ornstein-Uhlenbeck process.   
      
\vskip 10 pt   Write  
$\mathcal C=\{P_j,j\ge 1\}$, where $P_j$ are  the instants  of jumps  
of the random walk  $\{S_n , n\ge 1\}$,  
which are  recursively defined as follows, 
  \begin{equation}\label{C.Pn} P_1= \inf\{ n\ge 3: S_n=1\},\qq   P_{\nu +1} =\inf\{ n>P_\nu: S_n=1\}  \qq \quad \nu\ge 1  .
\end{equation}
  The main characteristic of Cram\'er's model is that heuristically 
   $\mathcal C $
  should imitate  well the sequence $\mathcal P$. He proved that with probability one, one has
\begin{equation}\label{C.Pngap} \limsup_{\nu \to \infty}\, \frac{P_{\nu +1}-P_{\nu}}{\log^2 P_{\nu }} =1 .\end{equation} 
   On the basis of this result, he wrote in \cite{C1} p.\,28, \lq\lq Obviously we may take this as a suggestion that, for the particular sequence of ordinary prime numbers $p_n$, some similar relation may hold.\rq\rq
\, 
He conjectured (Cram\'er's conjecture) that for some positive constant $c$,
\begin{equation}\label{C.pngap} \limsup_{\nu \to \infty}\, \frac{p_{\nu +1}-p_{\nu}}{\log^2 p_{\nu }} =c .\end{equation}
    
    The almost sure limit  result \eqref{C.Pngap} has  no arithmetical content, as  it is 
     purely probabilistic.
    Further the sequence of differences $\{P_{\nu +1}-P_{\nu},\nu\ge 1\}$ is a sequence of independent random variables, which is a very strong property. 
Impressive numerical evidences (up to $10^{18}$) of \eqref{C.pngap} are given  \cite{N}, see also \cite{G}, but depending on the scale of the observed phenomenon, $10^{18}$ might be a very little number (at least in the fast-growing hierarchy of   numbers), and paraphrasing Odlyzko's note \cite{O},  that conjecture, if true, can be just barely true. On the other hand, if the conjecture were true, no real singularity should  appear, in other words  the observed phenomenon is being from the beginning locally \lq similar\rq. If so, one may wonder what could be a reason. 
 We refer the reader to \cite{C2}, \cite{C1}, \cite{G}, \cite{P}, \cite{Se} notably, for results conforting or contrary to
the Cram\'er conjecture,  which nowadays  still 
 appears as a   mathematical  \lq spell\rq.

     \vskip 7 pt Cram\'er's model   does not assert that there are any primes in the sequence $\mathcal C$, and variants of this model either. An important question, apparently  overlooked in the related literature, thus concerns the possible primality of $S_n$, namely the study of the probability $\P \{S_n\  \text{prime}  \}$, prior to   the one of probability $\P \{P_{\nu}\  \text{prime}  \}$. The LIL (for instance) shows that   such a property is tightly related to the distribution of primes in small intervals,    making  thereby vain the hope of obtaining definitive results, 
        even on assuming RH.  Let   $m_n=\E S_n=\sum_{j=3}^n \frac1{\log j}$, $B_n={\rm Var }S_n= \sum_{j=3}^n \frac1{\log j}(1-\frac1{\log j})$. Let $b>1$.  The intervals 
$$ I_n=[m_n-\sqrt{2bB_n\log\log  n},m_n+\sqrt{2bB_n\log\log  n}],$$
are no longer overlapping as soon as $n$ runs along very moderated growing subsequences. Along such a subsequence $S_n$ can be prime, $n$ large, only if   $I_n$ contains a prime number. These intervals are of type
$$ [x-c\,x^{\frac{1}{2}}(\log\log  x)^{\frac{1}{2}}, x], $$
 for some $c>0$. It is at present quite out of reach, even on assuming the validity of the RH, 
to  decide for which $x$, 
 such an interval contains a prime number. One can also makes the similar observation from the sharpened version of the local limit theorem given in Proposition \ref{lltsharp[cramer]}.
 
  \vskip 3pt Thus we are in a case where we have a model predicting largest size of gaps between primes, whereas, even on  RH, we could not know whether  the \lq primes\rq\, of the model   are prime.  Recall some results on primes in small intervals. Assuming the RH, the best result known to us states as follows,
\begin{equation}\label{RH.pi}\pi(x)-\pi(x-y)\, =\, \int_{x-y}^x \frac{\dd t}{\log t}+ 
\mathcal O\Big( x^{\frac{1}{2}}\log \Big\{ \frac{y}{x^{\frac{1}{2}}\log x}\Big\}\Big)
\end{equation}
for $y$ in the range 
$ 2x^{\frac{1}{2}}\log x \le y\le x$. 
Thus for $M\ge 2$ fixed,
\begin{equation}\pi(x)-\pi(x-Mx^{\frac{1}{2}}\log x)
\sim\,  x^{\frac{1}{2}} \big\{M+ 
\mathcal O\big(  \log M\big)\big\}.\end{equation}
See Heath-Brown \cite{HB},  see also the recent paper \cite{HB1} and the references therein. Without assuming RH, Heath-Brown proved in \cite{HB} that if $\e(x)>0$, $\e(x)\to 0$ as $x\to\infty$, then 
\begin{equation}\pi(x)-\pi(x-y)\, =\,  \frac{y}{\log x}\Big(1+ \mathcal O( \e^4(x)) +\mathcal O\Big(\frac{\log\log x}{\log x}\Big)^4\Big)\end{equation}
for $y$ in the range 
 $  x^{\frac{7}{12}-\e(x)}  \le y\le \frac{x}{(\log x)^4}$. 
This slightly improves Huxley's earlier result in \cite{H} corresponding to $\e(x)=0$. Huxley's   result  shows that the PNT extends to intervals of the type $[x,x+x^\t]$,   $x^{\frac{7}{12} }  \le \t\le \frac{x}{(\log x)^4}$, namely that, 
\begin{equation}\label{huxley} \#\{[x,x+x^\t]\cap \mathcal P\} \sim \frac{x^\t}{ \log x }.
\end{equation} 

\vskip 7pt     We however obtain in Theorem \ref{Snprime.M},  without assuming RH, a  sharp estimate  of $\P \{S_n\ \text{prime}\, \}$, for almost all $n$, namely for all $n$, $n\to\infty$ through a set $\mathcal S$  of natural density $1$. The lower bound,  
\begin{equation}\label{abel.base,intro}  
\liminf_{  \mathcal S\ni n\to\infty}\ (\log n)\P \{S_n\ \hbox{prime}  \}
>0,
   \end{equation}
is also proved. \vskip 7pt
   Further the property for $S_n$ to be $z$-quasiprime
   is investigated.  
We   obtain in  Theorem \ref{cramer.quasi.prime.P} 
  a sharp estimate  of the probability that $S_k$  be $\zeta$-quasiprime,   $k$ large and  for the range of values $ \zeta_0\le  \zeta\le \exp\big\{ c {\log( k)}/{\log\log( k)}\big\}$, $c>0$.

\section{Main Results.} \label{1a}

It is well-known that the LIL (at least for centered square integrable i.i.d. sums) has slower amplitude than the one given by the classical normalizing factor $\sqrt{2n\log\log n}$, when $n$ is restricted to subsequences. For instance, if  $n$ runs along the subsequence
$\mathcal N=\{2^{2^k}, k\ge 1\}$, then the LIL restricted to $\mathcal N$ holds with normalizing factor
$ \sqrt{ 2n\log \log \log n}$. See \cite{W2} for a  characterization of the LIL for subsequences.
The  same phenomenon holds in fact - with no additional requirement - for the Cram\'er model. 

\begin{theorem}\label{cramer.LIL.sub} Let $\mathcal N$ be any increasing sequence of integers. Then, 
\begin{eqnarray*} \limsup_{\mathcal N\ni j\to \infty}\frac{|S_j- m_j|}{ \sqrt{ B_j}
\,\p_\mathcal N(j)}&=&1, \end{eqnarray*}
almost surely, where function $\p_{\mathcal N}(n)$ is defined in \eqref{cramer.phi.N}.
\end{theorem}

Roughly speaking, given $M>1$,  $I_k=]M^k,M^{k+1}]$,
$\p_{\mathcal N}(n)$ is defined as being equal to  $\sqrt{2\log (p+2)}$ if 
$n\in {\mathcal N}\cap I_{\k_p} $,   $I_{\k_p}$ being  the $p$-th   interval    intersecting  $\mathcal N$.

\vskip 10 pt 

In the next Theorems we obtain  very sharp results on        the length, 
   and also the  frequencies of the   intervals $I$ for which \eqref{cramer.small.ampl.} holds, namely 
\begin{equation*} \sup_{n\in I} \frac{|S_n-m_n|}{ \sqrt{B_n}}\le z,
\end{equation*}
$z$ a positive real  (corresponding to  $\p_\mathcal N\equiv{\rm Const}$).

\begin{theorem}\label{cramer.small.ampl.t1}
Let $f:[1,\infty)\to\R^+$ be     a non-decreasing function such that $f(t)\uparrow \infty$   with
$t$  and   $f(t)= o_\rho(t^\rho)$.  There exists a Brownian motion $W$ such that  
$$ \liminf_{k\to \infty}\sup_{e^k\le B_j\le e^kf(e^k)} \frac{|S_j-m_j|}{ 
\sqrt{B_j}} \,=\,\liminf_{k\to \infty}\sup_{e^k\le
B_j\le e^kf(e^k)}\frac{|W(B_j)|}{ \sqrt{B_j}}, 
$$
with probability $1$. 
Let $f_c(t)=  \log^c t $, $c>0$.   Then
$$\liminf_{k\to \infty}\sup_{e^k\le B_j\le e^kf_c(e^k)}\frac{|S_j-m_j|}{ 
\sqrt{B_j}}\le z, 
 $$
with probability $1$, if and only if $c\le 1/\l(z)$. And $\l(z)$ is   the smallest eigenvalue in the Sturm-Liouville
equation 
\begin{equation} \label{sl/e} \psi''(x)-x\psi'(x)= -\l \psi(x), \qq \psi(-z)=\psi(z)=0.
\end{equation}
This is   a positive strictly
decreasing continuous function of $z$ on $]0,\infty[$. Further, 
  \begin{equation} \label{sl/smei}\l(z)\sim \frac{\pi^2}{ 4z^2}\,, \qq\qq {\rm  as}\   z\to 0 .
\end{equation}

\end{theorem}

Towards this aim, we prove that Cram\'er's model   satisfies an invariance principle: 
\begin{proposition}[IP]\label{cramer.IP}
 Let $1/\a<\b<1/2$.  There exists   a Brownian motion $W$ such that if 
\begin{equation*}   \Upsilon=\sup_{n}\frac{1}{ B^\b_n} \, \sup_{j\le
n}{|S_j- m_j-W(B_j)| }    
\end{equation*}  
then,  
$\E  \Upsilon^{\a'}\!<\infty$, $0\le \a'<\a$. 
\end{proposition}
 Thanks to the IP above, the question studied in \eqref{cramer.small.ampl.} can be transferred into  a similar one concerning Brownian motion. This is done in section \ref{s3}, where    Theorem \ref{cramer.small.ampl.t1} is   proved.

 \vskip 3 pt
  We also obtain a sharp estimate on   the number of occurences of the sets
\begin{equation} \label{bkw} B_{ k}(f,z)=\Big\{\sup_{j\in J_{e^k}}
\frac{|S_j-m_j|}{\sqrt {B_j}}\le z\Big\} , \qq k=1,2,\ldots
\end{equation}
where $J_N= \{j:N\le B_j<Nf(N)\}$.
Let  $$A_k(f,z)=\Big\{\sup_{e^k\le t\le e^kf(e^k) } \frac{|W(t)|}{ \sqrt t}\le z\Big\},$$
  and let also $\nu _n(f,z)= \sum_{k=1}^n\P\{{A_k(f,z)}\}$.
   \begin{theorem}\label{cramer.frequencies.t1}  Let   $0<z'<z<z''$. Let also  $0<c \le 1/\l(z') $. Then for $a>3/2$,
 \begin{equation*}  \P\Big\{ \sum_{k=1}^n \chi_{B_k(f_c,z)}\le    \nu _n(f_c,z'')+ \mathcal O_a \Big(  \nu^{1/2} _n(f_c,z'')  
\log^a \nu _n(f_c,z'') \Big),   \quad \hbox{$n$ ultimately}\Big\}=1,
 \end{equation*}
  \begin{equation*}  \P\Big\{    \nu _n(f_c,z') \le   \sum_{k=1}^n \chi_{B_k(f_c,z)}+\mathcal O_a \Big(\, \nu^{1/2} _n(f_c,z')  
\log^a \nu _n(f_c,z')\Big),   \quad \hbox{$n$ ultimately}\Big\}=1.
 \end{equation*}
  Further  for all $n$,  
  $$K_1(z)\,\sum_{k=1}^n k^{-c\l(z)}  \le   \nu_n(f_c,z)  \le  K_2(z)\,\sum_{k=1}^n 
  k^{-c\l(z)}.$$
  \end{theorem}
Estimates of the sums $\sum_{k=1}^n k^{-c\l(z)}$ are given  in \eqref{re.zetasum}.
The question  
 arises whether   the refinements obtained (Theorems \ref{cramer.LIL.sub}, \ref{cramer.small.ampl.t1},  \ref{cramer.frequencies.t1})   may also have an interpretation  on the function $\pi(x)$.  
   \vskip 5  pt The Cram\'er  model is used to \lq predict\rq\, several, sometimes  quite elaborated results on the distribution of primes.
  The example  given in  \eqref{cramer.pnt1a} is 
very    striking, as the subsequence-LIL   is a well-known 
   companion result of the standard LIL, and cannot be dissociated from it. Thus if one uses the Cram\'er  model to make  such a prediction concerning the PNT (see after \eqref{cramer.pnt} and \eqref{cramer.pnt.LIL}) from the  standard LIL, probably its most simple prediction, one should also consider the prediction which arises with \eqref{cramer.pnt1a}, and argue whether this is another deficiency of the model or not.
    \vskip 2 pt 
The same sort of considerations is in order
   concerning the 
   frequency of  large gaps between \lq primes\rq .
   See   \eqref{cramer.Counting.Em} and after in Appendix \ref{s2.3}.

 \vskip 17 pt 
Concerning the probability that $S_n$  be prime or quasi-prime, and the primality of ${P_n}$, we prove the following results.  
\subsection{Primality of ${S_n}$.}
 
\begin{theorem}\label{Snprime.M} 
{\rm (i)} For any constant $b>1/2$, 
\begin{equation}\label{Snprime.i}   \P \{S_n\ \text{prime}\, \}
\, =\, \frac{ 1 }{ \sqrt{2\pi  B_n  } }\,\int_{m_n-\sqrt{ 2bB_n\log n}}^{m_n+\sqrt{ 2bB_n\log n}}
 e^{- \frac{(t- m_n)^2}{    2  B_n   } } \, \dd \pi(t)  
    + \mathcal O\Big( \frac{(\log n)^{3/2 }}{\sqrt n} \Big),
  \end{equation}
as $n\to\infty$.
  \vskip 3 pt {\rm (ii)} There exists a set of integers $\mathcal S$ of density 1, such that  \begin{equation}\label{abel.base,}  
 \P \{S_n\ \text{prime}\, \}
\, =\, \frac{ (1+ o( 1) )}{ \sqrt{2\pi  B_n  } } 
\int_{m_n-\sqrt{ 2bB_n\log n}}^{m_n+\sqrt{ 2bB_n\log n}}
  \,  e^{- \frac{(t- m_n)^2}{    2  B_n   } }\, \dd \pi(t),  \end{equation} 
as $n\to \infty$,  $n\in \mathcal S$. Further,
   \begin{equation}\label{abel.base,,}  
\liminf_{  \mathcal S\ni n\to\infty}\ (\log n)\P \{S_n\ \hbox{prime}  \}
\ge  \frac{1  }{  \sqrt{2\pi e}\, } .
   \end{equation}
 \end{theorem} 
The proof uses a result of  Selberg  \cite{Se}.
       \vskip 8 pt
 \subsection{Quasi-primality of ${S_n}$.}    \noi     Let $\Pi_z=\prod_{p\le z}p$. According to Pintz \cite{P}, an integer $m$ is   $z$-quasiprime, if $(m, \Pi_z)=1$. Let $S'_n= \sum_{j= 8}^n \xi_j$, $n\ge 8$. Note that the introduction of $S'_n$ in place of $S_n$ is not affecting Cram\'er's conjecture, see Remark \eqref{sna}. In the next Theorem we study for all $n$ large enough, the probability that $S'_n$   be $z$-quasiprime.

 \begin{theorem}\label{cramer.quasi.prime.P}
 We have for any $0<\eta<1$, and all $n$ large enough  and      $ \zeta_0\le  \zeta\le \exp\big\{ \frac{c\log n}{\log\log n}\big\}$, 
\begin{eqnarray*}  \P\big\{  S'_n\hbox{\ $\zeta$-quasiprime}\big\}   \,\ge  \, (1-\eta)   \frac{ e^{-\gamma} }{  \log
\zeta }. \end{eqnarray*}
 where $\gamma$ is  Euler's constant and $c$ is a positive constant.\end{theorem} 

The   approaches used to prove the above Theorems   not  apply to the study  of    the primality of    $P_\nu$. 

\vskip 3 pt  
 \subsection{Primality of ${P_n}$.} 
We show that when the \lq primes\rq~  $ P_\nu$   are observed   along    moderately growing subsequences,  then with probability 1, they  ultimately avoid       any given infinite set of primes satisfying a reasonable  tail's  condition.  We also test which infinite sequences of primes  are ultimately avoided  by   the \lq primes\rq~  $ P_\nu$, with probability 1. More precisely  
we     answer  the following   question: 
\begin{question}  {\it Given  an increasing sequence of naturals  $\mathcal K$ and increasing sequence of primes $\mathfrak P$,  under which conditions   is $\mathfrak P$ avoided  by all $P_\nu$,   $\nu$ large enough, $\nu\in\mathcal K$,  with probability 1}? 
\end{question}

\begin{theorem}\label{mathfrak}  Let $\mathcal K$ be an increasing sequence of naturals such that  
the series  $\sum_{k\in \mathcal K}  k^{-\b} $ converges for some   $\b\in ]0, \frac12[$.
Let $\mathfrak P$  be an increasing sequence of primes such that for some  $b>1$,
 $$\sup_{k\in\mathcal K} \frac{\#\{\mathfrak P\cap [k,  b k] \}}{  k^{\frac12-\b}}<\infty. $$
Then \begin{eqnarray*}
   \P\big\{   \D_k\notin\mathfrak P, \quad k\in \mathcal K\ \hbox{ultimately}\big\}=1.
    \end{eqnarray*}
Further, 
\begin{eqnarray*}
   \P\big\{   P_\nu\notin\mathfrak P, \quad \nu\in \mathcal K\ \hbox{ultimately}\big\}=1.
    \end{eqnarray*}
Moreover (case $\b=1/2$), let  $\mathfrak P$ be such that 
$\sum_{p\in \mathfrak P , \, p>y} p^{-1/2} = \mathcal O \big(y^{-1/2}\big)$, and $\mathcal K$ be   such that
  $\sum_{k\in \mathcal K}  k^{-1/2}<\infty$.
Then 
\begin{eqnarray*}
   \P\big\{   P_\nu\notin\mathfrak P, \quad \nu\in \mathcal K\ \hbox{ultimately}\big\}=1.
    \end{eqnarray*}
    \end{theorem} 

\vskip 3 pt

  \vskip 9 pt 
The paper is organized as follows.   The study of the   quasi-primality of $S_k$ is made in section \ref{7}, the one of the primality of $S_n$ occupies the whole section \ref{5}, and the one of the primality of   $P_n$ is made in section \ref{4}. These sections are a forming the main body of the  paper.   Theorem \ref{cramer.LIL.sub} is proved in section \ref{s4}.
In section \ref {s3},  we prove the IP,  as well as Theorem \ref{cramer.small.ampl.t1}, after some preliminary background, and Theorem 
\ref{cramer.frequencies.t1}.    The standard limit theorems for the Cram\'er model, statements and proofs,    and various results (sharp estimate of the characteristic function of $S_n$, divisors of Bernoulli sums) used in the course of the proofs  are moved  to the Appendix \ref{appendix-1}.  Some remarks concerning Cram\'er's proof are also included.


\section{Primality of $  {S_n}$: Proof of Theorem \ref{Snprime.M}.}\label{5} 
We need a sharper form of the local limit theorem for $S_n$ than the one given in     Lemma \ref{l1cramer}. 
\begin{proposition}\label{lltsharp[cramer]} We have the following estimate
\begin{eqnarray*}   \Big| \P \{S_n =\kappa \} -\frac{ e^{- \frac{(\k- m_n)^2}{    2B_n    } } }{ \sqrt{2\pi B_n    }}     \Big|         & \le &    C\,
 \frac{(\log n)^{3/2}}{n}  ,
   \end{eqnarray*}
for all $\k\in\Z$ such that
$$ |\k- m_n|  \le C \,\frac{n^{3/4}}{\log n} .$$ 
\end{proposition}  
The remainder term is of order $\mathcal O(\frac{(\log n)^{3/2}}{n}) $, which is much better than $o(   (\frac{\log n}{n})^{1/2})  $ in Lemma \ref{l1cramer}. This is a consequence of Corollary 1.11 in \cite{GW}. For the reader convenience we recall it. Introduce first the necessary notation. Let
$v_{0} $ and $D >0$  be   real numbers. We denote by   $\mathcal L(v_{ 0},D )$ the lattice  defined by the
sequence $v_{ k}=v_{ 0}+D k$, $k\in \Z$.  We associate to any    random variable $X$ taking values in  ${\mathcal L}(v_0,D)$ with probability one,   the following characteristic,
\begin{eqnarray}\label{vartheta}  \t_X =\sum_{k\in \Z}\P\{X=v_k\}\wedge\P\{X=v_{k+1}\} ,
\end{eqnarray}
 where $a\wedge b=\min(a,b)$. Note   that
     $
\t_{X }<1$. 

\begin{lemma}[\cite{GW}, Cor.\,1.11]
\label{ger1} Let  $ X_1,\ldots ,X_n$  be  independent   random variables taking almost surely values in a common lattice   $\mathcal L(v_{
0},D )=\{v_k,k\in \Z\}$, where $v_{ k}=v_{ 0}+D k$, $k\in \Z$,   $v_{0} $ and $D >0$ are   real numbers. We assume that 
\begin{equation}\label{basber}  \t_{X_j}>0, \qq \quad  j=1,\ldots, n.
\end{equation}
Let $S_n=X_1+\ldots +X_n$.
Let $\psi:\R\to \R^+$ be even, convex and such that   $\frac
{\psi(x)}{x^2}$  and $\frac{x^3}{\psi(x)}$  are non-decreasing on $\R^+$. We  assume that
  \begin{equation}\label{did}  \E \psi( X_j )<\infty  , \qq \quad  j=1,\ldots, n  .
 \end{equation} 
 
 Put $$L_n=\frac{  \sum_{j=1}^n\E \psi (X_j)  }
{   \psi (\sqrt
{ {\rm Var}(S_n )})}  .$$  \vskip 2 pt
  
   Let  $0<\t_j\le \t_{X_j}$ and denote $\Theta_n=\sum_{j=1}^n \t_j$. Further assume that $\frac{  \log \Theta_n }{\Theta_n}\le  {1}/{14} $. Then, for all $\k\in \mathcal L( v_{
0}n,D )$ such that
$$\frac{(\k- \E
S_n)^2}{    {\rm Var}(S_n)  } \le \sqrt{\frac{ \Theta_n}{14 \log \Theta_n} },$$  we have
\begin{eqnarray*}   \Big| \P \{S_n =\kappa \} -\frac{ D e^{- \frac{(\k- \E
S_n)^2}{    2 {\rm Var}(S_n)    } } }{ \sqrt{2\pi {\rm Var}(S_n)     }}     \Big|         & \le &    C_3\Big\{
D\Big(\frac{    {   \log \Theta_n }     }{
 {    {\rm Var}(S_n)   \Theta_n} } \Big)^{1/2}  +    \frac{    L_n
 +  \Theta_n^{-1}
}{ \sqrt{   \Theta_n} } \Big\} .
   \end{eqnarray*}
And $C$ is an absolute constant.
\end{lemma}

   \begin{proof}[Proof  of Proposition 
  \ref{lltsharp[cramer]}] In  our case $D=1$. Further for $j=3,\ldots , n$, $\P\{\xi_j=k\}\wedge\P\{\xi_j=k+1 \}=\frac{1}{\log j+2}$, if $k=0$,  and equals $0$ for $k\in \Z_*$. Thus $  \t_{\xi_j}=\frac{1}{\log j}$.  We choose $ \t_j= \t_{X_j}$,  $\psi(x)=|x|^3$. Then  $\Theta_n=\E S_n\sim \frac{n}{\log n}$, ${\rm Var}(S_n)=B_n \sim 
  \frac{n}{\log n} $, $L_n\sim( \frac{\log n }{   n})^{1/2}$.

Thus ($m_n=\sum_{k=1}^n    \frac{1}{\log k}$, $B_n=\sum_{k=1}^n  (1-\frac{1}{\log k})(\frac{1}{\log k})$)
\begin{eqnarray*}   \Big| \P \{S_n =\kappa \} -\frac{ e^{- \frac{(\k- m_n)^2}{    2  B_n    } } }{ \sqrt{2\pi  B_n  }}     \Big|         & \le &    C\,
 \frac{(\log n)^{3/2}}{n}  ,
   \end{eqnarray*}
for all $\k\in\Z$ such that
$ |\k- m_n|  \le C \,\frac{n^{3/4}}{\log n}$. 
\end{proof}

\vskip 20 pt 
 
\begin{proof}[Proof of Theorem \ref{Snprime.M}.]
(i) By Lemma 7.1 p.\,240 in \cite{P}, for $0\le x\le B_n$
\begin{eqnarray*} \P \{|S_n-m_n|\ge x \}&=& \P \{S_n-m_n\ge x \}
+\P \{-(S_n-m_n)\ge x \}
\cr &\le &2\,\exp\Big\{-\frac{x^2}{2B_n}\Big(1-\frac{x}{2B_n}\Big)\Big\},
\end{eqnarray*}
noticing that   $\{-\xi_j\}_j$ also satisfies the conditions of Kolmogorov's Theorem. Let $b>b'>1/2$.
Then for all sufficiently large $n$, since $\log  B_n\sim \log n$,
\begin{equation} \P \{|S_n-m_n|\ge  \sqrt{2bB_n\log   n} \}\le
2\,  n^{-b'}
.
\end{equation}

 We have \begin{align*}   \Big|  \P \{S_n \in\mathcal P \}&  -  \P \{S_n \in \mathcal P\cap[m_n- \sqrt{ 2bB_n\log n},m_n+\sqrt{ 2bB_n\log n}] \} \Big| 
\cr &\le    \P \{|S_n-m_n|\ge \sqrt{ 2bB_n\log n} \}
\cr &\le   n^{-b'}.
   \end{align*}
Further,
\begin{eqnarray*}& & \Big|  \P \{S_n \in \mathcal P\cap[m_n-\sqrt{ 2bB_n\log n},m_n+\sqrt{ 2bB_n\log n}] \}
\cr & &-\sum_{\k\in\mathcal P\cap[m_n-\sqrt{ 2bB_n\log n},m_n+\sqrt{ 2bB_n\log n}]} \frac{ e^{- \frac{(\k- m_n)^2}{    2  B_n   } } }{ \sqrt{2\pi  B_n  } } 
\Big|
 \cr &\le & \sum_{\k\in\mathcal P\cap[m_n-\sqrt{ 2bB_n\log n},m_n+\sqrt{ 2bB_n\log n}]}\Big|\P \{S_n =\kappa \} -\frac{ e^{- \frac{(\k- m_n)^2}{    2  B_n   } } }{ \sqrt{2\pi B_n   } }     \Big|   
  \cr     & \le &    C\,\#\big\{\mathcal P\cap[m_n-\sqrt{ 2bB_n\log n},m_n+\sqrt{ 2bB_n\log n}] \big\}\cdot
 \frac{(\log n)^{3/2}}{n} 
  \cr     & \le & 
  C\, \sqrt b\ \frac{(\log n)^{3/2 }}{\sqrt n} 
.
  \end{eqnarray*}
Therefore
\begin{equation} \label{3} \Big|  \P \{S_n \in \mathcal P  \}
 -\sum_{\k\in\mathcal P\cap[m_n-\sqrt{ 2bB_n\log n},m_n+\sqrt{ 2bB_n\log n}]} \frac{ e^{- \frac{(\k- m_n)^2}{    2  B_n   } } }{ \sqrt{2\pi  B_n  } } 
\Big|
     \, \le \, 
   C\,\sqrt b\  \frac{(\log n)^{3/2 }}{\sqrt n} 
.
  \end{equation}
 By    expressing the inner sum as a Riemann-Stieltjes integral  \cite[p.\,77]{A}, we get

\begin{equation}\label{abel.base}   \P \{S_n \in \mathcal P  \}
\, =\, \int_{m_n-\sqrt{ 2bB_n\log n}}^{m_n+\sqrt{ 2bB_n\log n}}
  \,\frac{ e^{- \frac{(t- m_n)^2}{    2  B_n   } } }{ \sqrt{2\pi  B_n  } }\, \dd \pi(t)
    + \mathcal O\Big( \frac{(\log n)^{3/2 }}{\sqrt n} \Big)
.
  \end{equation}

\vskip 5 pt (ii) We note that 
  \begin{eqnarray} \label{int.min}
  \int_{m_n-\sqrt{ 2bB_n }}^{m_n+\sqrt{ 2bB_n }}
  \,\frac{ e^{- \frac{(t- m_n)^2}{    2  B_n   } } }{ \sqrt{2\pi  B_n  } } \dd \pi(t)&\ge & L \  \frac{ \pi(m_n+\sqrt{ 2bB_n })-\pi(m_n-\sqrt{ 2bB_n })}{ \sqrt{   B_n  }},
\end{eqnarray}
with $L=\frac{e^{-b }}{  \sqrt{2\pi }}$. 
\vskip 3 pt
 We use a well-known result of Selberg     \cite[Th.\,1]{Se}. Let    $\Phi(x)$ be positive and increasing and such that $\frac{\Phi(x)}{x} $  decreasing for $x>0$. Further assume that 
 \begin{eqnarray} \label{phi.ab}
{\rm (a)}\quad \lim_{x\to\infty}  \frac{\Phi(x)}{x}= 0 \qq\qq {\rm (b)}\quad\liminf_{x\to\infty}\frac{\log \Phi(x)}{\log x}>\frac{19}{77}.
\end{eqnarray}
Then there exists a (Borel measurable) set $\mathcal S$   of positive reals of density one such that
 \begin{equation}  \label{phi.ab.enonce}\lim_{\mathcal S\ni x\to \infty}\frac{\pi(x+\Phi (x))-\pi(x)}{({\Phi(x)}/{\log x})}  =1.
\end{equation}
 Let $\Phi(x) = \sqrt{2b  x}$. Then the   requirements in \eqref{phi.ab} are fulfilled, and so \eqref{phi.ab.enonce} holds true. 
  Now let $C$ be some possibly large but fixed positive number, as well as some positive real $\d<1/2$.    
    \vskip 5 pt   By \eqref{phi.ab.enonce}, the  set of $x>0$, call it $\mathcal S_\d$, such   that 
    \begin{equation}\label{pi.phi.d}
\pi(x+\Phi(x))-\pi(x) \, \ge  \,(1-\d)\, \frac{ \Phi(x)}{\log x}.
 \end{equation}
  has    density 1. Note that if $\d'<\d''$ then $\mathcal S_{\d'}\subseteq \mathcal S_{\d''}$.   
 \vskip 3 pt Pick  $x\in\mathcal S_\d$ and let $\D(x)= \pi(x+\Phi(x))-\pi(x)$. Note that if $|y-x|\le C$,   $\big|\Phi(y)-\Phi(x)\big|=o(1)$ for $x$   large.
 Thus  for every $y\in [x-C,x+C]$, $\big|\D(y)-\D(x)\big|\, \le C'$, and so 
$$\D(y)   \, \ge  \,(1-\d)\, \frac{ \Phi(x)}{\log x}-C'.$$
 the constant $C'$ depending on $C$ only. 
 As $\big|\frac{\Phi(x)}{\log x}- \frac{\Phi(y)}{\log y}\big|\le 
\frac{C}{2\sqrt x \log x}$, we have 
\begin{equation*}
\D(y)   \, \ge  \,(1-\d)\, \frac{ \Phi(y)}{\log y}-C' -\frac{C}{2\sqrt x \log x}.
 \end{equation*} 
  
Thus every $y\in [x-C,x+C]$ also satisfies \begin{equation}\label{pi.phi.d.a}
\D(y)   \, \ge  \,(1-2\d)\, \frac{ \Phi(y)}{\log y},
 \end{equation} 
 if  $x $ is large enough.
 
\vskip 3pt   
Let $\nu=\nu(x)$ be the  unique integer   such that $m_{\nu-1}< x\le m_{\nu}$. As $m_{\nu}-m_{\nu-1}=o(1)$,  $\nu \to \infty$,  it follows that $m_{\nu}\in [x-C,x+C]$ provided that   $x$ is large enough, in which case we have by   \eqref{pi.phi.d.a},
  \begin{equation}\label{pi.phi.d.b}
 \frac{\pi(m_{\nu}+\Phi(m_{\nu}))-\pi(m_{\nu})}{\Phi(m_{\nu})}  \ge  \frac{1-2\d}{\log m_{\nu}}.
 \end{equation} 
   Let $X\ge 1$ be a large positive integer and $\e$  a small positive real. The number $N(X)$ of intervals $]\m-1,\m]$, $\m\le X$ such that $\mathcal S_\d \cap ]\m-1,\m]\neq \emptyset$ verifies $N(X)/X\sim 1$, $X\to \infty$, since $\mathcal S_\d$ has density 1. 
 
 Given such an $\m\le X$, pick $x\in \mathcal S_\d \cap ]\m-1,\m]$. 
 We know (recalling that $m_{\nu}-m_{\nu-1}=o(1)$,  $\nu \to \infty$) that some $m_{\nu}$, $\nu=\nu(x)$ belongs to   $ ]\m-1-\e,\m+\e]$, and that \eqref{pi.phi.d.b} is satisfied.   The union of these intervals $[\m-1-\e,\m+\e ]$ is contained in $[1-\e , X+\e ]$. It follows that the number of $\nu$ such that \eqref{pi.phi.d.b} is satisfied, forms a set of density 1. 
  
 \vskip 4 pt  We now use an induction argument in order to replace $2\d$ in  \eqref{pi.phi.d.b} by a quantity $\e(\nu)$ which  tends to 0 as $\nu$ tends to infinity  along some other set of density 1, which we shall build explicitly.    Let $\mathcal T_n$ be the set   of $\nu$'s of density 1, corresponding to $\d=\frac1n$, $n\ge 3$. Let $X_3$ be large enough so   that $\#\{\mathcal T_3\cap [1, X ]\}\ge X (1-1/3)$ for all $X\ge X_3$. Next let  $X_4>X_3$ be sufficiently large  so   that $\#\{\mathcal T_4\cap [X_3, X ]\}\ge X (1-1/4)$ for all $X\ge X_4$. Like this we manufacture an increasing sequence $X_j$,   verifying for all $j\ge 3$,
  $$\#\{\mathcal T_j\cap [X_{j-1}, X ]\}\ge X (1-1/j), \qq \qq {\rm for\ all}\  X\ge X_j .$$
  The resulting set 
  $$ \mathcal T= \bigcup_{j=3}^\infty \mathcal T_j\cap [X_{j-1}, X_j ]$$
  has density 1 and further we have the inclusions 
  $$\mathcal T\cap [X_{l-1}, \infty )=\bigcup_{j=l}^\infty \mathcal T_j\cap [X_{j-1}, X_j ]\subset \mathcal T_l\cap\bigcup_{j=l}^\infty   [X_{j-1}, \infty )=\mathcal T_l\cap [X_{l-1},\infty),\qq\quad l\ge 4,$$
as   the sets $\mathcal T_j$ are decreasing with $j$ by definition.   

\vskip 3 pt We finally have  by \eqref{int.min},
  \begin{equation}\label{pi.phi.d.b.1}
 \frac{\pi(m_{\nu}+\Phi(m_{\nu}))-\pi(m_{\nu})}{\Phi(m_{\nu})}  \ge  \frac{1-\e(\nu)}{\log m_{\nu}},
 \end{equation}
along $\mathcal T$, for some sequence of reals $\e(\nu)\downarrow 0$ as $\nu\to \infty$.
 
\vskip 3 pt Therefore 
  \begin{eqnarray}\label{base} 
  \int_{m_\nu-\sqrt{ 2bB_\nu }}^{m_\nu+\sqrt{ 2bB_\nu }}
  \,\frac{ e^{- \frac{(t- m_\nu)^2}{    2  B_\nu   } } }{ \sqrt{2\pi  B_\nu  } } \dd \pi(t)&\ge & L \  \frac{ \pi(m_\nu+\sqrt{ 2bB_\nu })-\pi(m_\nu-\sqrt{ 2bB_\nu })}{ \sqrt{   B_\nu  }}\cr & \ge &  L\, \frac{1-\e(\nu)}{\log m_{\nu}},
\end{eqnarray}
for all $\nu \in\mathcal T$, recalling that $L=\frac{e^{-b }}{  \sqrt{2\pi }}$.

Also  \begin{equation}\label{abel.base,proof}  
\liminf_{  \mathcal T\ni \nu\to\infty} (\log \nu)\P \{S_\nu\ \hbox{prime}  \}
\ge  \frac{1  }{ \sqrt{2\pi e}\, } .
 \end{equation}
  It further follows from \eqref{abel.base} that  
\begin{equation}\label{abel.base,,proof}  
 \P \{S_\nu\ \hbox{prime}  \}
\, =\, \big(1+ o( 1)\big)
\int_{m_\nu-\sqrt{ 2bB_\nu\log \nu}}^{m_\nu+\sqrt{ 2bB_\nu\log \nu}}
  \,\frac{ e^{- \frac{(t- m_\nu)^2}{    2  B_\nu   } } }{ \sqrt{2\pi  B_\nu  } }\, \dd \pi(t),
  \end{equation}
   all $\nu\in \mathcal T$. This achieves the proof of Theorem \ref{Snprime.M}.
   \end{proof}
\vskip 12  pt 
Some remarks: estimate \eqref{phi.ab.enonce} extends the PNT to smalls intervals $[x,x+\Phi (x)]$ for almost all $x$. Selberg (developing Cram\'er's first results \cite{C2}) proved with Theorem 4 in \cite{Se} a much stronger result since an error term is provided. Assuming the RH, he proved  that  for any fixed $\vartheta>0$, \eqref{huxley} is true for almost all $x$. This is an easy  consequence of the very sharp Theorem 1 in \cite{Se}.
 The approach used, as well as the   alternate approach in Richards \cite{Ri}, seem not allow one to treat the   question whether there exists a version of the PNT with an error term valid for almost all integers. 
 This question is in relation with the one on the sensitivity of the error term to subsequences (cf. Introduction, Theorems \ref{cramer.LIL.sub}, \ref{cramer.small.ampl.t1}).

\section{Quasi-primality of ${S_n}$: Proof of Theorem \ref{cramer.quasi.prime.P}.}\label{7} 

Theorem \ref{cramer.quasi.prime.P} is   a direct consequence of a more general result, which we shall prove now. 
Let     $2<\l_1<\l_2<\ldots$ be an increasing sequence of reals. 
Let $\{\zeta_j,j\ge 1\}$ be a sequence of independent binomial random variables defined by $\P\{\zeta_j=1\}= \frac{1}{\l_j}=1-\P\{\zeta_j=0\}$.  
   
\begin{theorem}\label{cramer.quasi.prime} Let $T_k=\sum_{j=1}^k  \zeta_j$, $k\ge 1$. Assume that  $\m_k=  \sum_{j=1}^k \l_j^{-1}\uparrow\infty$ with  $k$. For any $0<\d<1$, we have for any $0<\eta<1$, and all $k$ large enough and   $ \zeta_0\le  \zeta\le \exp\big\{ \frac{\log(2\d \m_k)}{\log\log(2\d \m_k)}\big\}$, 
\begin{eqnarray*}  \P\big\{  T_k \hbox{ is $\zeta$-quasiprime}\big\}   &\ge& (1-\eta) \,  \frac{ e^{-\gamma} }{  \log
\zeta },\end{eqnarray*} 
  where $\gamma$ is  Euler's constant and $c$ is a positive constant.
\end{theorem} 

  The case $\l_j=\log (j+2)$, $j\ge 8$ corresponds to  the Cram\'er model, and we have in particular the following
  \begin{corollary}\label{cramer.quasi.prime.cor}
We have for any $0<\eta<1$, and all $n$ large enough  and      $ \zeta_0\le  \zeta\le \exp\big\{ \frac{c\log n}{\log\log n}\big\}$, 
\begin{eqnarray*}  \P\big\{  S'_n\hbox{\ $\zeta$-quasiprime}\big\}   \,\ge  \, (1-\eta)   \frac{ e^{-\gamma} }{  \log
\zeta }. \end{eqnarray*}
 \end{corollary} 

\vskip 3 pt
 The proof is based on a randomization argument, and uses  
 the  Lemma  below.
\begin{lemma}\cite[Theorem 2.3]{di} \label{md}
Let $X_1, \dots, X_k$     be independent random variables, with $0 \le X_j \le 1$ for each $j$.
Let $Y_k = \sum_{j=1}^k X_j$ and $\mu = \E Y_k$. For any $\e >0$,
 \begin{eqnarray*}
(a) &&
 \P\big\{Y_k \ge  (1+\e)\mu\big\}
    \le  e^{- \frac{\e^2\mu}{2(1+ \e/3) } } .
\cr (b) & &\P\big\{Y_k \le  (1-\e)\mu\big\}\le    e^{- \frac{\e^2\mu}{2}}.
 \end{eqnarray*}
\end{lemma}

\begin{proof}[Proof of Theorem \ref{cramer.quasi.prime}] Let $\{\e_j,j\ge 1\}$ be a sequence of independent Bernoulli random variables. Let $\{ \check{\zeta_j},j\ge 1\}$ be another sequence of independent random variables, which is  independent from the sequence $\{ \e_j,j\ge 1\}$, and such that $\zeta_j \stackrel{{\rm a.s.}}{=}\check{\zeta_j} \,\e_{j}$ for all $j$.     Let  $\check{\E}$ (resp. $\check{\P}$) denote the conditional expectation  (resp. conditional probability) with respect to the $\s$-field generated by the sequence $\check{\zeta_j}$, $j\ge 1$. Write $ 
    T_k =\,\sum_{j=1}^k  \check{\zeta_j} \,\e_{j}$.
We have 
 \begin{equation}\label{block products.primality}
   \P\big\{P^-( T_k ) > \zeta\big\}  \,=\,
  \check{\E}   \P\big\{P^-\big(  \sum_{j=1}^k  \check{\zeta_j} \,
 \e_j\big)> \zeta\big\}
     \,=\, \check{\E}  \P\big\{P^-\big( B_{\sum_{j=1}^k  \check{\zeta_j}}\big)> \zeta\big\}.
  \end{equation}
According to Theorem \ref{Bernoulli.quasi.p}, there exist a positive real $c $ and  positive constants $C_0$, $\zeta_0$ such that for $k$ large
enough   we have,
\begin{equation}\label{P-Bn,} 
\Big|\hskip1pt\P\big\{ P^-( B_k ) > \zeta \big\}-     \frac{ e^{-\gamma} }{  \log
\zeta }\,\Big|  \le  \frac{ C_0 }{   \log^2 \zeta } \qq \qquad (\,\zeta_0\le  \zeta\le k^{c/\log\log k}\,) .
\end{equation}

Let $0<\d< 1$ and set
 \begin{eqnarray*}   A_k=\Big\{  \exp\Big\{ \frac{c\log(\sum_{j=1}^k  \check{\zeta_j})}{\log\log(\sum_{j=1}^k  \check{\zeta_j})}\Big\}    \ge \zeta\Big\},\qq\quad  C_k=\Big\{ \sum_{j=1}^k  \check{\zeta_j}> 2\d \m_k\Big\} .\end{eqnarray*}
Assume that 
$$ \zeta_0\le  \zeta\le \exp\Big\{ \frac{c\log(2\d \m_k)}{\log\log(2\d \m_k)}\Big\}
.$$
On $C_k$,
$$  \exp\Big\{ \frac{c\log(\sum_{j=1}^k  \check{\zeta_j})}{\log\log(\sum_{j=1}^k  \check{\zeta_j})}\Big\}> \exp\Big\{ \frac{c\log(2\d \m_k)}{\log\log(2\d \m_k)}\Big\}\ge \zeta,
$$
and so $C_k\subset A_k$. 
Therefore on $C_k$,
$$\Big| \P\big\{P^-\big( B_{\sum_{j=1}^k  \check{\zeta_j}}\big)> \zeta\big\}-     \frac{ e^{-\gamma} }{  \log \zeta }\,\Big|  \le  \frac{ C_0  }{   \log^2 \zeta }.$$
We have
 \begin{eqnarray*}  \P\big\{P^-( T_k ) > \zeta\big\}  &=& \check{\E}  \P\big\{P^-\big( B_{\sum_{j=1}^k  \check{\zeta_j}}\big)> \zeta\big\}\,\ge\, \check{\E}  \chi\{C_k\}\,\P\big\{P^-\big( B_{\sum_{j=1}^k  \check{\zeta_j}}\big)> \zeta\big\}
 \cr &\ge& \big( \frac{ e^{-\gamma} }{  \log
\zeta }- \frac{ C_0 }{    \log^2 \zeta }  \big)\,\P\{C_k\} .\end{eqnarray*} 
Consequently, for $ \zeta_0\le  \zeta\le \exp\Big\{ \frac{c\log(2\d \m_k)}{\log\log(2\d \m_k)}\Big\}$,
\begin{eqnarray*}  \P\big\{P^-( T_k ) > \zeta\big\}   &\ge& \big( \frac{ e^{-\gamma} }{   \log
\zeta }- \frac{ C_0 }{    \log^2 \zeta }  \big)\,\P\{C_k\}. \end{eqnarray*} 

We apply  Lemma \ref{md} with $X_j=\check {\zeta_j}$. The random variables  $\check{\zeta_j}$ are independent and  verify  $\P\{ \check{\zeta_j} =1\} =1-\P\{ \check{\zeta_j} =0\}= 2\l_j^{-1}$. Further let $\check{\m}_k =\E \sum_{j=1}^k \check{\zeta_j}  =2\sum_{j=1}^k \l_j^{-1}$. By assumption $\sum_{j=1}^k \l_j^{-1}\uparrow\infty$ with  $k$. 
\vskip 3 pt 
Let $0< \rho<1$. By Lemma \ref{md}, 
\begin{equation}\label{Ck.min} \P\Big\{ \sum_{j=1}^k  \check{\zeta_j}\le \d\,\check{\m}_k \Big\}\le e^{- \frac{(1-\d)^2\check{\m}_k}{2}} \le \rho, 
\end{equation}
for all $k\ge k_\rho$ say. Thus
\begin{equation}\label{Ck.min,}\P\{C_k\}=\P\Big\{  \sum_{j=1}^k  \check{\zeta_j}> \d\check{\m}_k  \Big\} \ge 1-\rho. 
\end{equation}
We therefore arrive at
\begin{eqnarray*}  \P\big\{T_k \hbox{ is $\zeta$-quasiprime}\big\}   &\ge& (1-\rho) \big( \frac{ e^{-\gamma} }{   \log
\zeta }- \frac{ C_0 }{    \log^2 \zeta }  \big)\,  \end{eqnarray*} 
for any $ \zeta_0\le  \zeta\le \exp\big\{ \frac{c\log(2\d \m_k)}{\log\log(2\d \m_k)}\big\}$. As $\rho $ can be a small as we wish, we can state that for any given $0<\d<1$, we have for any $0<\eta<1$, and all $k$ large enough and any $ \zeta_0\le  \zeta\le \exp\big\{ \frac{c\log(2\d \m_k)}{\log\log(2\d \m_k)}\big\}$, 
\begin{eqnarray*}  \P\big\{T_k \hbox{ is $\zeta$-quasiprime}\big\}   &\ge& (1-\eta) \,  \frac{ e^{-\gamma} }{   \log
\zeta }. \end{eqnarray*} 
 
\end{proof} 

\begin{theorem}\label{Bnprime.nk}  Let $(n_k)_{k\ge 1}$ be an increasing sequence of integers such that,  
\begin{equation*} 
\sum_{k\ge 1}    \frac{   \log\log n_k }{      \log  n_k} \ <\ \infty\, .
\end{equation*}
Then, \begin{equation*} 
\P\big\{   T_{n_k}  \, \hbox{not prime, $k$ ultimately} \big\}\ =\ 1.
\end{equation*}
In particular, 
\begin{equation*} 
\P\big\{   S'_{n_k}  \, \hbox{not prime, $k$ ultimately} \big\}\ =\ 1.
\end{equation*}
\end{theorem}

\begin{proof} We also have
\begin{eqnarray} 
   \P\big\{T_k\ \hbox{prime}\big\}  &=& 
  \check{\E}   \P\big\{\sum_{j=1}^k  \check{\zeta_j} \, \e_j\ \hbox{prime}\big\}\,=\, \check{\E}  \P\big\{B_{\sum_{j=1}^k  \check{\zeta_j}}\ \hbox{ prime}\big\}
    \end{eqnarray}
By Corollary \ref{Bnprime.ubound} in Appendix \ref{appendix-3}, there exists an absolute constant $C_1$ such that for all $n$ large enough,
\begin{equation*} 
\P\big\{   B_n  \ \hbox{prime} \big\}\le C_1\,    \frac{   \log\log n }{   c  \log  n},
\end{equation*}
($c$ is the same constant as in Theorem \ref{Bernoulli.quasi.p}). This along with \eqref{Ck.min}, imply 
    \begin{eqnarray}\label{block products.primality.a}
   \P\big\{T_k\ \hbox{prime}\big\}  &=& 
  \check{\E}  \P\big\{B_{\sum_{j=1}^k  \check{\zeta_j}}\ \hbox{prime}\big\}
     \cr &\le & 
     \check{\P}  \{\sum_{j=1}^k  \check{\zeta_j}\le \d\,\check{\m}_k\big\}+
     C_1\, \check{\E} \chi\{\sum_{j=1}^k  \check{\zeta_j}>\d\,\check{\m}_k\}    \frac{   \log\log \sum_{j=1}^k  \check{\zeta_j}}{   c  \log  \sum_{j=1}^k  \check{\zeta_j}}
 \cr &\le & 
     e^{- \frac{(1-\d)^2\check{\m}_k}{2}}+
     C_1\,    \frac{   \log\log \d\,\check{\m}_k}{   c  \log \d\,\check{\m}_k}
\cr &\le &
     C(c,\d)\,    \frac{   \log\log \d\,\check{\m}_k}{   c  \log \d\,\check{\m}_k},
 \end{eqnarray}
 for all $k\ge \k(c,\d)$. 
    Theorem \ref{Bnprime.nk} now  follows from Borel-Cantelli lemma.
 \end{proof}

 \section{Primality of $  {P_n}$.}\label{4} 
 
 \subsection{The inclusion $ {\mathcal C\subset  \mathcal B}$.}  We use the fact (Introduction) that on a possibly larger probability space, $\xi_j=\check{\xi_j} \,\e_{j}$,  $j\ge 8$, almost surely,   where $\{\e_j,j\ge 8\}$ is a sequence of independent Bernoulli random variables and $\{ \check{\xi_j},j\ge 8\}$,   a sequence of independent binomial random variables  which is  independent from the sequence $\{ \e_j,j\ge 8\}$. This is well defined since    $2/\log j<1$, if $j\ge 8$. 
   The  indices such that $\check{\xi_j} \,\e_j=1$ are obviously contained in the set of indices    such that $\e_j=1$. So that if   $\mathcal C_1=\{ j\ge 8\!:\,\xi_j=1\}$, $\mathcal B=\{ j\ge 1\!:\e_j=1\}$ and $\mathcal B_1=\{ j\ge 8\!:\,\e_j=1\}$,   the inclusion
  \begin{equation}\label{inc1} \mathcal C_1\subset \mathcal B_1,
  \end{equation}
is satisfied with probability 1. Whence also,
\begin{proposition}\label{C1B1}   
The   inclusion 
$\mathcal C\subset  \mathcal B$ holds true with positive probability.  
\end{proposition}
\subsection{Instants of jump in the Bernoulli model.}  The instants of jump of the  sequence $ \{B_k,k\ge 1\}$  are defined as follows: Put  $\d_0=0$, $\D_0=0$ and for any integers
$\ell\ge 1$, $k\ge 1$,
\begin{equation}
  \d_\ell  =\inf \{ n\ge 1\!:\, \e_{n+\d_1+\dots+\d_{\ell-1}}=1 \},\qq   \D_k  =\d_1+\dots+\d_{k }.
 \end{equation} 
 We have $\mathcal B:=\{ j\ge 1\!:\e_j=1\}=\{\D_k,k\ge 1\}$. 
 Proposition \ref{C1B1} suggests to first study the probability that $\D_k$ be prime. We first establish some necessary properties of $\d_\ell$ and $\D_k$. 

\begin{theorem} \label{deltak.prop} {\rm (i)} The random variables $\d_k$ are   i.i.d., exponentially distributed,
$$   \hbox{\it $\P\{\d_\ell
=m\}=2^{-m}$ for all $\ell\ge 1$ and $m\ge 1$.}
$$
{\rm (ii)}   Then $\E \d_1 =2$, ${\rm Var}(\d_1)=2$. Further,
\begin{eqnarray}\label{char.xi} \E e^{2i\pi \d_1 t}\,=\,\sum_{m=1}^\infty \frac{e^{2i\pi mt}}{2^m}, 
 \qq \E e^{2i\pi \D_k t}\,=\,
  \sum_{\nu=1}^\infty  \frac{e^{2i\pi \nu t}}{2^\nu}\,C^{k-1}_{\nu -1}
 .
\end{eqnarray} 
 {\rm (iii)}  For all $k\ge 1$, $m\ge 1$,
$\P\{\D_k=m\} = \frac{C^{k-1}_{m -1}}{2^m}=\frac12 \P\{B_{m -1}=k-1\}.
$
\vskip 2 pt\noi {\rm (iv)} With probability one,
 \begin{eqnarray*}  
\limsup_{n\to \infty}\frac {\d_k }{ \log  k}\,=\,1 ,
\qq  
\qq \limsup_{n\to \infty}\frac {\D_k - 2k }{2\sqrt{  k\log\log  k}}&=&1 
  .
\end{eqnarray*}
 {\rm (v)}  The  central limit theorem and the  local limit theorem holds, namely $ \frac{\D_k - 2k}{\sqrt{2 k}}\, \stackrel{ {\mathcal D}}{ \Rightarrow}  \, \mathcal N(0,1)$, as $k\to\infty$ and 
 \begin{equation*}   \sup_{n }\Big|  \sqrt k \P\{\D_k=n\}-\frac{1}{   2\sqrt{  \pi} }e^{-
\frac{(n-2k  )^2}{   4 k  } }\Big| ={\mathcal O}\big(k^{-1/2}\big).
 \end{equation*} 
\end{theorem}
  \begin{proof} (i)   As 
    $\{\d_1=\m\}=\{\b_1=\ldots=\b_{\m-1}=0,  \b_\m=1\}\in \s(\b_1,\ldots,\b_{\m})$, we   have $ \P\{\d_1=\m\}=\P\{\b_1=\ldots=\b_{\m-1}=0\,,\, \b_\m=1\}=2^{-\m}$. By recursion on $\ell\ge 1$, one establishes
  \begin{align*}
 &\P\{\d_{\ell+1} =m\}=
    \sum_{m_1=1}^\infty\ldots \sum_{m_\ell=1}^\infty \P\Big\{ \d_1= m_1\,,\, \ldots\,,\,\d_\ell= m_\ell\,,\,
   \cr &\qq \qq \qq\b_{m_1+\ldots +m_\ell+1}=\ldots=  \b_{m_1+\ldots +m_\ell+m-1} =0\,,\, 
    \b_{m_1+\ldots +m_\ell+m}=1\Big\}
 \cr & = \sum_{m_1=1}^\infty\ldots \sum_{m_\ell=1}^\infty  \prod_{i=1}^\ell \P\big\{ \d_i= m_i\big\}\ \P\Big\{\b_{m_1+\ldots +m_\ell+1}=\ldots=  \b_{m_1+\ldots +m_\ell+m-1} =0\,,\, 
 \cr &\qq\qq\qq\qq  \b_{m_1+\ldots +m_\ell+m}=1\Big\}     
 \, = \sum_{m_1=1}^\infty\ldots \sum_{m_\ell=1}^\infty  2^{-m_1-\ldots-m_\ell}\, 2^{-m}\,= 2^{-m},
 \end{align*}
and   also    find that
$  \P\{\d_1=a_1\,,\, \ldots\,,\, \d_{\ell +1}=a_{\ell +1}\}
 =\prod_{i=1}^{\ell +1}\P\{\d_i=a_i \}$,  
for all positive integers $a_i$, $1\le i\le \ell +1$. Further $\{\d_1=m_1\,,\, \d_2=m_2\,,\,\ldots\,,\, \d_{\ell +1}=m_{\ell +1}\}\in \s(\b_1,\ldots,\b_{m_1+\ldots +m_{\ell +1}})$. 

\vskip 2 pt 
(ii) The characteristic function of $\d_1$ being $ \E e^{2i\pi \d_1 t}=\sum_{m=1}^\infty \frac{e^{2i\pi mt}}{2^m}$, we have  \begin{equation}\label{char.func}
\E e^{2i\pi \D_k t}= \Big(\sum_{m=1}^\infty \frac{e^{2i\pi mt}}{2^m}\Big)^k\, =\,
  \sum_{m_1=1}^\infty\ldots  \sum_{m_k=1}^\infty\frac{e^{2i\pi (m_1+\ldots+ m_k)t}}{2^{m_1+\ldots+ m_k}}
\, =\,  \sum_{\nu=k}^\infty \frac{e^{2i\pi \nu t}}{2^\nu}\, C^{k-1}_{\nu -1}
.\end{equation}

\vskip 5 pt 
   Let $S(u)=\sum_{ a=0 }^\infty   e^{-au }$, then   
$S(u)  =\frac{1}{  1-e^{-  u }}$, $S'(u)= -  \frac{ e^{- u}}{  (1-e^{- u})^2}=- \sum_{ a=1 }^\infty   a e^{-au }$, $S''(u)=    \frac{ e^{- u}(1+e^{- u})}{ 
(1-e^{- u})^3} =- \sum_{ a=1 }^\infty   a^2 e^{-au }$.  Thus first and second moments of $\d_1$ can be computed and one finds that $\E \d_1=2$, ${\rm Var}(\d_1)=2$. 
\vskip 3 pt 
(iii) Follows from  
\begin{eqnarray}
\P\{\D_k=m\}&=&\sum_{\stackrel{m_i\ge 1\,,\, 1\le i\le k}{m_1+\ldots+ m_{k}=m}} \P\{\d_1=m_1,\ldots ,\d_{k}=m_{k} \}
\cr &=& \sum_{\stackrel{m_i\ge 1\,,\, 1\le i\le k}{m_1+\ldots+ m_{k}=m}} \frac{1}{2^{m_1+\ldots +m_k}}\,=\,\frac{q_k(m)}{2^m}\,=\,\frac{C^{k-1}_{m -1}}{2^m}
\,=\,  \frac12 \P\{B_{m -1}=k-1\} .
\end{eqnarray}

    (iv)  Immediate.
 
\vskip 3pt
  (v)  The  central limit theorem is obvious since the $\d_k$'s are i.i.d. and square integrable. 
   By Theorem 6 p.197  in \cite{P}, as  $\E |\d_1|^3<\infty$, we have
   \begin{equation} \label{alfa}   \sup_{n }\Big|  \sqrt k \P\{\D_k=n\}-\frac{1}{   2\sqrt{  \pi} }e^{-
\frac{(n-2k  )^2}{   4 k  } }\Big| ={\mathcal O}\big(k^{-1/2}\big).
 \end{equation}

\end{proof}

\vskip 3pt
We see   that the divisors of 
the     jump's instants $ \D_k$ admit a    simple formulation. In particular, we have from (iii),
\begin{equation}\label{div.delta.k}\P\big\{   \D_k
\,\hbox{ prime} \big\} \ = \   
 \frac12 \sum_{\stackrel{\nu \ge k}{\nu  \, {\rm prime}}}\P\{ B_{\nu -1} =k-1\}.
 \end{equation} 
 
The formula $ \sum_{v=0}^\infty C^z_{v+z} x^v= \frac{1}{(1-x)^{z+1}}$, valid for $|x|<1$, further implies  
\begin{eqnarray}  \frac12  \sum_{\nu \ge k }\P\{ B_{\nu -1} =k-1\}=1 . 
  \end{eqnarray}  

\subsection{Proof of Theorem \ref{mathfrak}}
 {\rm (i)} By the local limit theorem  for Bernoulli sums
     \begin{eqnarray}\label{lltber} \sup_{z}\, \Big|  \P\big\{B_n=z\} 
 -\sqrt{\frac{2}{\pi n}} e^{-\frac{ (2z-n)^2}{ 
2 n}}\Big| =o\Big(\frac{1}{n^{3/2}}\Big). 
  \end{eqnarray} 
Besides by  Theorem \ref{deltak.prop}-(iii), $\P\{\D_k=m\} = \frac12 \P\{B_{m -1}=k-1\}
$. Thus 
\begin{eqnarray*}
   \P\big\{   \D_k\in
\mathfrak P 
\big\}&=& \sum_{\stackrel{\nu \ge k}{\nu\in
\mathfrak P}}\P\{\D_k =\nu\}\,=\, 
   \frac12 \sum_{\stackrel{\nu \ge k}{\nu\in
\mathfrak P}}\P\{ B_{\nu -1} =k-1\} 
   \cr &=&
   \sum_{\stackrel{\nu \ge k}{\nu\in
\mathfrak P}}\frac{1}{\sqrt{2\pi (\nu -1)}} \,
    e^{-\frac{ (2k-\nu -1)^2}{  2(\nu-1)}}+o\Big(\sum_{\stackrel{\nu \ge k}{\nu\in
\mathfrak P}}\frac{1}{ \nu^{3/2}}\Big)
 \cr   &=&  \sum_{\stackrel{\nu \ge k}{\nu\in
\mathfrak P}}\frac{1}{\sqrt{2\pi (\nu -1)}} 
    e^{-\frac{ (2k-\nu -1)^2}{  2(\nu-1)}}+ o \big(\frac{1}{\sqrt k}\big) .
   \end{eqnarray*}

 Obviously $\sum_{\nu \ge 3k}\frac{1}{\sqrt{   \nu }} 
    e^{-\frac{ (2k-\nu -1)^2}{  2(\nu-1)}}\,\le \,   
    e^{-C_1k} $.
Now by assumption, for $k\in\mathcal K$,
 $$\sum_{\stackrel{k\le \nu \le  b k }{ \nu\in
\mathfrak P}}\frac{1}{\sqrt{\nu}} 
    e^{-\frac{ (2k-\nu -1)^2}{  2(\nu-1)}} \,\le \,C\, \frac{\#\{\mathfrak P\cap [k,  b k] \}}{\sqrt k}\,\le \,C\, k^{- \b}.$$
  It easily follows that  
$$ \P\big\{   \D_k\in
\mathfrak P 
\big\}\le  \,C_b\, k^{- \b},\qq \quad   k\in\mathcal K.$$   
The series   $\sum_{k\in \mathcal K}   \P\big\{   \D_k\in
 \big\} <\infty$ converges. It follows from Borel-Cantelli lemma that
\begin{eqnarray*}
   \P\big\{   \D_k\notin\mathfrak P, \quad k\in\mathcal K\ \hbox{ultimately}\big\}=1.
    \end{eqnarray*}
 \vskip 3 pt 
To prove the same assertion concerning the sequence $\mathcal P$, it suffices to argue as before Proposition \ref{C1B1}, by considering the sequence $\{\check{\xi_j} \,\e_{j},j\in \mathcal K\}$.


\section{Instants of small amplitude in the Cram\'er model: Proofs.}\label{s3}

  Before giving the proofs of Theorem \ref{cramer.small.ampl.t1}, Proposition \ref{cramer.IP} and Theorem \ref{cramer.frequencies.t1},  it is necessary for the understanding of the matter to recall some notation and results from \cite{W}. Let $f:[1,\infty)\to\R^+$ be here and throughout   a non-decreasing function such that $f(t)\uparrow \infty$   with
$t$  and   $f(t)= o_\rho(t^\rho)$. 
    The intervals $I$ considered (cf.\,\eqref{cramer.small.ampl.}) are of type 
 $[  e^k ,   e^k  f(e^k )]$, $k=1,2,\ldots$. 
 \vskip 3 pt 
 Put
\begin{equation} \label{akw} A_k(f,z)=\Big\{\sup_{e^k\le t\le e^kf(e^k) } \frac{|W(t)|}{  \sqrt t}<z\Big\} , \qq k=1,2,\ldots
\end{equation}
  Let $  U(t)= W(e^t)e^{-t/2}, t\in \R $ be the
Ornstein-Uhlenbeck process. It will be more convenient to work with $U$ instead of $W$. Observe
that$$  A_k(f,z)  = \Big\{\sup_{k\le s\le k +\log f(e^k) }|U(s)|\le
z\Big\}.    $$
And so  as $U$ is   stationary 
$$  \P\{A_k(f,z) \} = \P\Big\{\sup_{0\le s\le \log f(e^k) }|U(s)|\le
z\Big\}.    $$   
  \vskip 3 pt  
\noi  We say that
$f\in
\mathcal U_z$ whenever
 $\P\big\{ \limsup_{k\to \infty}A_k(f,z)  \big\}=0 $, and that $f\in \mathcal V_z$ if
 $\P\big\{ \limsup_{k\to \infty}A_k(f,z)  \big\}=1 $.  By the   $0$-$1$ law (since $U$ is strongly mixing), the latter probabilities
can only be 
$0$ or
$1$.   
  
Notice that if
$f\in
\mathcal U_z$, then with probability one 
$$J(f ): = \liminf_{k\to \infty}
\sup_{k\le s\le k +\log f(e^k)}|U(s)|\,\ge\,z, $$
whereas  $J(f)\le z$, almost surely if $f\in \mathcal V_z$. Introduce also for $n=1,2,\ldots$ the   counting function
$$N_n(f,z)= \sum_{k=1}^n\chi_{A_k(f,z)},  $$ 
with corresponding  mean  
$$   \nu _n(f,z):= \E N_n(f,z).$$
 
\smallskip\par  The classes $\mathcal U_z$ and $\mathcal V_z$ have been characterized in \cite[Th.\,1.1]{W}, where  the following   simple convergence criterion is proved.

\begin{theorem} \label{fr1}   
Let $\Sigma(f) = \sum_k
f(e^k)^{- \l(z)}$, where $\l(z)>0$ is defined in Theorem \ref{cramer.small.ampl.t1}. Then
 $$f\in \mathcal U_z\quad (\hbox{resp. $\in \mathcal V_z$})\quad \Longleftrightarrow \quad \Sigma(f) <\infty\quad
 (\hbox{resp. $=\infty$}).$$
 Further if $\Sigma(f)=\infty$,  for any $a>3/2$,  
$$N_n(f ,z) \,=\, \nu _n(f ,z)+ \mathcal O \big(  \nu^{1/2} _n(f ,z)  
\log^a \nu _n(f ,z) \big),$$
with probability one.
   And  there are positive constants $K_1(z),K_2(z)$ depending on $z$ only, such that for all $n$  
$$K_1(z)  \le  \frac{ \nu _n(f ,z)}{  \sum_{k=1}^n  f(e^k)^{- \l(z)}} \le  K_2(z).$$
  \end{theorem}
  
\smallskip\par
 In \cite{W}, applications to the
Kubilius model are given. The class of functions $f_c(t)=  \log^c t $, $c>0$, is of  special interest  in view of such applications. For these  functions, Theorem \ref{fr1} implies
\begin{corollary} \label{J}    If $  c  > 1/\l(z) $, then  $f_c\in \mathcal U_z$ whereas  $f_c\in \mathcal V_z$ if $0<c  \le 1/\l(z) $.
Further, for any
$0<c  \le 1/\l(z) $ and
$a>3/2$, 
 \begin{eqnarray*} N_n(f_c,z)&\stackrel{a.s.}{=}&\nu _n(f_c,z)+ \mathcal O \Big(  \nu^{1/2} _n(f_c,z)  
\log^a \nu _n(f_c,z) \Big) 
 . 
\end{eqnarray*}
 And for all $n$,  $K_1(z)  \le  \frac{ \nu _n(f_c,z)}{  \sum_{k=1}^n k^{-c\l(z)}}  \le  K_2(z)$.
   \end{corollary}

    Accordingly, if 
\begin{equation}\label{Iw} I(f)  :=\liminf_{k \to \infty} \sup_{e^k\le t\le e^kf(e^k) } \frac{|W(t)|}{ 
\sqrt t}, 
\end{equation}
  then $\P\{  I(f_c) \le z \}=1$ if and only if $0<c  \le 1/\l(z) $.  This is clear  in view of (\ref{akw}).
Noticing  that $I(f)\le I(g)$ whenever $f(N)\le g(N)$ for all $N$ large,  we therefore also deduce 
\begin{corollary}\label{W} We have 
  $\P\{   I(f_c)\le z \}=1$  if and only if    $0<c  \le 1/\l(z) $.
And $\P\{  I(f) =\infty \}=1$  if $f (t)\gg_c   f_c(t)$ for all $c$.
\end{corollary}

We  need a suitable invariance principle for sums of independent random
variables, which will be also used in the next section. This one is due to Sakhanenko (see \cite{S}, Theorem 1).

 \vskip 3 pt  Let $\{\xi_j, j\ge
1\}$ be independent centered random variables    with absolute second moments. Let $t_k= \sum_{j=1}^k \E \xi_j^2$,  
 $S_k=  \sum_{j=1}^k  \xi_j  $ and let $\{r_k, k\ge 1\}$ be some non-decreasing sequence of positive reals. Let $\a\ge 2$, $y>0$. Put 
successively,
\begin{eqnarray} \D_n&=&\sup_{k\le n}|S_k-W(t_k) |,\cr 
\D & = & \sup_{n\ge 1} \frac{\D_n}{  r_n},\cr 
\overline{\xi} &=& \sup_{j\ge 1} \frac{|\xi_j|}{  r_j},\cr
L_\a(y)&=& \sum_{j\ge 1} \E \min\Big\{\frac{|\xi_j|^\a}{  y^\a r_j^\a}, \frac{|\xi_j|^2}{  y^2 r_j^2}\Big\}.
\end{eqnarray}
 \begin{lemma} \label{sha1}There exists an absolute constant $C$ such that for any fixed $\a$, there exists a Brownian motion $W$ such that
for all
  $x>0$,
$$ \P\big\{\D  \ge C \a x\big\}\le L_\a(x) .  $$
\end{lemma}
 
 In our case, 
we choose $X_i= \xi_i-\E \xi_i$.
 Let $1/\a<\b<1/2$. Take $r_j=(\sum_{i=1}^j \E |X_i|^2)^{\b}=B_j^{ \b}$. Since $B_j\asymp \frac{j}{\log j}$, $j\to \infty$, it follows 
$$\sum_{j\ge 1} \frac{\E |X_j|^\a}{  r_j^{\a }}=\sum_{j\ge 1} \frac{\E |X_j|^\a}{  B_j^{  \a\b}}\le C \sum_{j\ge 2} \frac{1}{  ( j/ \log j )^{\a\b}(\log j)}<\infty,$$
as $\a\b>1$. Thus
$$L_\a(y)\le y^{-\a}\sum_{j\ge 1} \E  \frac{|X_j|^\a}{    r_j^\a} \le C_\a\,y^{-\a}.
$$
  By Lemma \ref{sha1}, there exists   a Brownian motion $W$ such that
  for all
  $x>0$,
$$ \P\big\{\sup_{n}\frac{1}{  r_n}  \sup_{j\le
n}{|S_j- m_j-W(B_j)| }   \ge C \a x\big\}\le C_\a\,x^{-\a}.  $$
By a simple use of Tchebycheff's inequality,    letting  
 $$  \Upsilon=\sup_{n}\frac{1}{  r_n}  \sup_{j\le
n}{|S_j- m_j-W(B_j)| },    $$
we deduce that 
 $$\E  \Upsilon^{\a'}  <\infty,\qq\quad (\a'<\a).$$
\vskip 3 pt 
We shall now   use Theorem \,1.5 in \cite{W}.
 We first note that   $\E X_j^2 = \frac{1}{\log j} (1-\frac{1}{\log j} )$, and for $\a>2$, as $|X_j|$ is bounded by $1+\frac{1}{\log j}\le C$, say, we have $\E |X_j|^\a\le C^{\a-2} \E X_j^2$. Further $\E |X_j|^\a\ge ( \E |X_j|^2)^{\a/2}\ge \big(\frac{C}{\log j}\big)^{\a/2}$. Thus
\begin{equation}\label{alpha}  v:= \sup_{j\ge 1} \frac{\E |X_j|^\a}{  \E |X_j|^2}<\infty.
\end{equation}
 Therefore  assumption (1.4) of Th.\,1.5 in \cite{W} is fulfilled. We deduce that 
there exists a Brownian motion $W$ such that 
$$ \liminf_{k\to \infty}\sup_{e^k\le B_j\le e^kf(e^k)}\frac{|S_j|}{  \sqrt{B_j}}\, =\,\liminf_{k\to \infty}\sup_{e^k\le
B_j\le e^kf(e^k)}\frac{|W(B_j)|}{  s_j}, 
$$
 with probability $1$. 
  By Corollary \ref{W},  $$\liminf_{k\to \infty}\sup_{e^k\le B_j\le e^kf_c(e^k)}\frac{|S_j|}{  \sqrt{B_j}}\le z, 
 $$
 with probability $1$, if and only if $c\le 1/\l(z)$.

\vskip 3 pt 
 \vskip 10 pt
We pass to the proof of Theorem \ref{cramer.frequencies.t1}.
  Let      $N$ be for the moment unspecified.  Then,
 \begin{eqnarray}\label{control.approx}\Big|  \sup_{j\in J_N}\frac{|S_j-m_j|}{\sqrt {B_j}} - \sup_{j\in J_N}\frac{|W(B_j)|}{\sqrt {B_j}} \Big|&\le & \Big(\sup_{j\in J_N}\frac{1}{ B_j^{1-2\b}}\Big)\, \Upsilon \, \to 0\, , 
 \end{eqnarray}
as $N$ tends to infinity, with probability one.  

 Recall   that
  $$A_k(f,z)=\Big\{\sup_{e^k\le t\le e^kf(e^k) } \frac{|W(t)|}{  \sqrt t}\le z\Big\},$$
  by \eqref{akw} and 
that $\nu _n(f_c,z)= \sum_{k=1}^n\P\{{A_k(f_c,z)}\}$. Further by Corollary \ref{J}, for all $n$,  
  $$K_1(z)\,\sum_{k=1}^n k^{-c\l(z)}  \le   \nu_n(f_c,z)  \le  K_2(z)\,\sum_{k=1}^n 
  k^{-c\l(z)},$$
and \begin{equation}\label{re.zetasum}
\sum_{k=1}^n k^{-c\l(z)}  =
\begin{cases}\frac{n^{  1-c\l(z)} }{ 1-c\l(z)} + \zeta(c\l(z)) + \mathcal O\left(n^{ -c\l(z)   }\right) &\qq \hbox{if $0<  c\l(z) <1,
$}
\cr \log n + \g +\mathcal O(\frac1n) &\qq \hbox{if $c\l(z)=1$},
 \end{cases}
\end{equation}
where $\g$ is Euler's constant, recalling that 
$\zeta(s) = \lim_{x\to\infty} \big(\sum_{n\le x} \frac1{n^s}- \frac{x^{1-s}}{1-s}\big)$, $0<s<1$. 
\vskip 3 pt
Choose $N=e^k$, $k=1,2,\ldots$ and write now more simply $B_{e^k}(f,z)=B_k(f,z)$.    Let also $0<z'<z<z''$. 

By \eqref{control.approx}, on a measurable set of probability as close to one as we please,  call it $\O^*$, we have for $k$ large enough, $k\ge k_0$ say,
 \begin{equation}A_k(f,z' )\subset B_k(f,z)\subset A_k(f,z'').
 \end{equation}
  As $\l(z)$   is   a positive strictly
decreasing continuous function of $z$ on $]0,\infty[$, we have $0<\l(z'')<\l(z)<\l(z')$. 
  
  \vskip 3 pt Let $f=f_c$  with $0<c \le 1/\l(z) $, and note that 
   $c   \le 1/\l(z'')$. Let $a>3/2$. By using Corollary \ref{J}, we get that on $\O^*$, for all $n$ large enough,
   \begin{equation}  \sum_{k=k_0}^n \chi_{B_k(f_c,z)}\le \sum_{k=k_0}^n\chi_{A_k(f_c,z'')}
\, =\, \nu _n(f_c,z'')+ \mathcal O_a \Big(  \nu^{1/2} _n(f_c,z'')  
\log^a \nu _n(f_c,z'') \Big).
 \end{equation}

We deduce that 
  \begin{equation}\label{occurences.ub} \P\Big\{ \sum_{k=1}^n \chi_{B_k(f_c,z)}\le    \nu _n(f_c,z'')+ \mathcal O_a \Big(  \nu^{1/2} _n(f_c,z'')  
\log^a \nu _n(f_c,z'') \Big),   \quad \hbox{$n$ ultimately}\Big\}=1.
 \end{equation}
 Similarly,  on $\O^*$ for all $n$ large enough,
   \begin{eqnarray}   \sum_{k=k_0}^n\chi_{A_k(f_c,z')}
\le \sum_{k=k_0}^n \chi_{B_k(f_c,z)} .
 \end{eqnarray}
Assume that $0<c \le 1/\l(z') $. By  using Theorem \ref{fr1}, we get
 \begin{equation}\label{occurences.lb} \P\Big\{ \sum_{k=1}^n \chi_{B_k(f_c,z)}\ge    \nu _n(f_c,z')-\mathcal O_a \Big(\, \nu^{1/2} _n(f_c,z')  
\log^a \nu _n(f_c,z')\Big),   \quad \hbox{$n$ ultimately}\Big\}=1.
 \end{equation}
  Note that \eqref{occurences.ub} is also valid if $0<c \le 1/\l(z') $.
 \vskip 3 pt 
 Hence we have proved Theorem \ref{cramer.frequencies.t1}.


\section{Subsequence LIL results for the Cram\'er model: Proofs.}\label{s4}
Let
${ \mathcal N}=
\{n_k, k\ge 1\}$ be any increasing sequence of integers and $M>1$; the
value of $M$ will be irrelevant. Let $I_0=]0,M]$ and for each integer 
$k\ge 1$, let $ I_k=]M^k,M^{k+1}]$. The subsequence of intervals 
$I_k$ such that  $I_k\cap {\mathcal  N}\not= \emptyset$, determines an 
increasing sequence of indices, which we denote by $\k= \{\k_p,p\ge 1\}$.
For any $n\in {\mathcal N}$ we put 
\begin{equation}\label{cramer.phi.N}
\p_{\mathcal N}(n)  =\sqrt{2\log (p+2)} \qquad \hbox{if} 
\quad n\in {\mathcal N}\cap I_{\k_p}.
\end{equation}
Recall that $r_j= B_j^\b$ and that $$  \Upsilon=\sup_{n}\frac{1}{  r_n}  \sup_{j\le
n}{|S_j- m_j-W(B_j)| }  .  $$

Now let $j_p^*= \max\{ j:  B_j\in\mathcal N\cap ]2^{p-1} , 2^p]\}$. As   
 
We have
\begin{eqnarray*}& & \Big|
\sup_{\stackrel{2^{p-1} <B_j\le 2^{p }}{ j\in\mathcal N}} 
\frac{|S_j- m_j|}{  \sqrt{ B_j}
\,\p(j)}
- \sup_{\stackrel{2^{p-1} <B_j\le 2^{p }}{  j\in\mathcal N}} 
\frac{|W(B_j)|}{  \sqrt{ B_j}\,\p(j)}\Big|
\cr & \le & 
\sup_{\stackrel{2^{p-1} <B_j\le 2^{p }}{ j\in\mathcal N}}\frac{|S_j- m_j-W(B_j)|}{  \sqrt{ B_j}\,\p(j)}
\,=\,\sup_{\stackrel{2^{p-1} <B_j\le 2^{p }}{  j\in\mathcal N}} \frac{|S_j- m_j-W(B_j)|}{ 
B_j^{1/2-\b}B_j^{\b }\,\p(j)}
\cr &\le & 
\Big(\sup_{\stackrel{2^{p-1} <B_j\le 2^{p }}{  j\in\mathcal N}}  \frac{1}{  B_j^{1/2-\b}\,\p(j)}\Big) \sup_{\stackrel{2^{p-1} <B_j\le 2^{p }}{  j\in\mathcal N}}\frac{|S_j- m_j-W(B_j)|}{   r_j } 
\cr &\le & 
\Big(\sup_{\stackrel{2^{p-1} <B_j\le 2^{p }}{  j\in\mathcal N}} 
 \frac{1}{  B_j^{1/2-\b}\,\p(j)}\Big)\, \frac{2^\b}{  r_{j_p^*}}\ \sup_{ j\le j_p^*}\,|S_j- m_j-W(B_j)| 
\cr &\le & 
2^\b\,\Big(\sup_{\stackrel{2^{p-1} <B_j\le 2^{p }}{  j\in\mathcal N}}  \frac{1}{  B_j^{1/2-\b }\,\p(j)}\Big)\,\cdot\Upsilon  
\ \  \to \   0  ,
 \end{eqnarray*}
as $p\to \infty$ almost surely, since $\b<1/2$. 
 
Therefore
\begin{eqnarray*} \limsup_{\mathcal N\ni j\to \infty}\frac{|S_j- m_j|}{  \sqrt{ B_j}
\,\p(j)}&=&\limsup_{\mathcal N\ni j\to \infty}\frac{|W(B_j)|}{  \sqrt{ B_j}\,\p(j)} , \end{eqnarray*}
almost surely. 

Now that
\begin{eqnarray*}  \limsup_{\mathcal N\ni j\to \infty}\frac{|W(B_j)|}{  \sqrt{ B_j}\,\p(j)}&=&1 , \end{eqnarray*}
almost surely, follows from the proof of Theorem 3.3 in \cite{W2}. This is rather easy to observe from estimates (3.22), (3.23), (3.24) in \cite{W2}.

 \vskip 15 pt 
\noi {\it Bibliographic Notes.}  We mainly refered during this work, to Cram\'er \cite{C2}, \cite{C1}, Granville's  exposition in \cite{G},  Pintz's systematic analysis of Cram\'er's model in  \cite{P},   Ellison's seminar paper
 \cite{E}, 
 which notably contains a detailed proof of Hoheisel's seminal result \cite{Ho},
  Ingham \cite{I},  Selberg \cite{Se}, Richards \cite{Ri}  where an interesting alternative approach to derive some of Selberg's results, under a weaker form, is purposed.


   \appendix

\section{Some characteristics of the Cram\'er model.}   
\label{appendix-1} 
\subsection{Classical limit theorems.}  \label{s2.1} Let us first briefly look at standard limiting results such as CLT, LIL and ASLLT which are fulfilled  by the Cram\'er model. 
 \begin{lemma}\label{l1cramer}  The SLLN, CLT,  LIL  and  LLT hold true. More precisely, 
\begin{eqnarray*}\cr{\rm (SLLN)} & & \lim_{n\to \infty}\frac {S_n }{m_n}=1, \ \  
\qq  \text{almost\ surely},
\cr {\rm (CLT)} & & \frac {S_n -m_n}{\sqrt B_n}\ \ \stackrel{{\mathcal D}}{\Rightarrow} \ \ 
{\mathcal N}(0,1), 
\qq  n\to \infty,
\cr{\rm (LIL)} & & \limsup_{n\to \infty}\frac {S_n -m_n}{\sqrt{2 B_n\log\log B_n}}=1, \ \  
\qq  \text{almost\ surely},
\cr {\rm (LLT)} & & \lim_{n\to \infty} \sup_{k} \Big|\sqrt{B_n} {\mathbb P}\{S_n=k\} -\frac{1}{\sqrt{2\pi}}e^{-  \frac{(k-m_n)^2}{2B_n }}\Big|=0.
\end{eqnarray*}
 \end{lemma}
 
 This is an immediate consequence of the following Lemma.
  \begin{lemma}\label{l2cramer} Let $\{X_j,j\ge 1\}$ be independent binomial random variables with $\P\{X_j=0\}=1-\P\{X_j=1\}= p_j$, for all $j$ and let $S_n= \sum_{j=1}^nX_j$,  $n\ge 1$. Further let $\s_j^2  ={\rm Var}( X_j)=p_j(1-p_j)$, $B_n={\rm Var}( S_n)= \sum_{j=1}^n p_j(1-p_j)$. 
  \vskip 2 pt 
 Assume that the series $\sum_{j }  p_j$ diverges and that $p_j=o(1)$. Then the SLLN, CLT,  LIL  and  LLT hold true.
 \end{lemma}
 
 \begin{proof} Let $\e>0$. Then,
 \begin{align*}\frac{1}{B_n}  \sum_{j=1}^n\ & \E \Big(\big|X_j -\E X_j\big|^2 \cdot \chi\big\{|X_j -\E X_j|>\e \,{\rm Var} (S_n)^{1/2}\big\} \Big)
 \cr & = \frac{1}{B_n}\sum_{j=1}^n  \E \Big(\big|X_j -p_j\big|^2 \cdot \chi\big\{|X_j -p_j|>\e \,\big( \sum_{j=1}^n p_j(1-p_j)\big)^{1/2}\big\}\Big).
 \end{align*}
 As the series $\sum_{j }  p_j$ diverges and $p_j=o(1)$, the above summands are $0$ for all $n\ge n_\e$, say. This implies that 
 $$\lim_{n\to \infty } \frac{1}{B_n}\sum_{j=1}^n   \E\Big( \big|X_j -\E X_j\big|^2 \cdot \chi\big\{|X_j -p_j|>\e \,{\rm Var} (S_n)^{1/2}\big\}\Big) = 0,$$
 for any positive $\e$. Whence the Lindeberg condition is satisfied, and so 
the CLT holds:
\begin{equation}\frac {S_n -m_n}{\sqrt B_n}\ \ \stackrel{{\mathcal D}}{\Rightarrow} \ \ {\mathcal N}(0,1), 
\qq  n\to \infty.
\end{equation} 
  
Concerning the LIL,     Kolmogorov's condition   that $X_n =o\big(\big( {B_n}/{\log\log B_n}\big)\big)^{1/2}$, almost surely, is trivially satisfied. Thus  the LIL directly follows from Kolmogorov's theorem.  See Petrov \cite{P}, Th. 7.1 p.\,239. The LIL implies the SLLN, given to the assumptions made. 
For establishing the LLT, we  apply Davis and McDonald's theorem \cite[Th.\,1.1]{MD} which we recall.

 \begin{theorem}\label{dmd.th}  Let $\{ X_j , j\ge 1\}$ be independent, integer valued random variables with partial sums
$S_n= X_1+\ldots +X_n$ and let
$f_j(k)=
{\mathbb P}\{X_j=k\}$.  Also for each $j$ and $n$, let
$$q(f_j)= \sum_{k} [f_j(k)\wedge f_j(k+1)], \qq Q_n=\sum_{j=1}^n q(f_j)  $$
and assume that $q(f_j)>0$ for each $j\ge 1$. Further assume that there exist numbers $b_n>0$, $a_n$  such that
 $$\lim_{n\to \infty}b_n= \infty, \qq \limsup_{n\to \infty} \ {b_n^2}/{Q_n}<\infty,$$
and
$$ \frac{S_n-a_n}{b_n} \ \ \stackrel{{\mathcal D}}{\Longrightarrow}\ \ {\mathcal N}(0,1).$$
  Then
$$ \lim_{n\to \infty} \sup_{k} \Big|b_n {\mathbb P}\{S_n=k\} -\frac{1}{\sqrt{2\pi}}e^{-  \frac{(k-a_n)^2}{2b_n^2}}\Big|=0. $$
 \end{theorem}
 With the notation used, $f_j(k)=
{\mathbb P}\{X_j=k\} $  and so 
$$q(f_j)= \sum_{k} \big(f_j(k)\wedge f_j(k+1)\big)=\big(f_j(1)\wedge f_j(0)\big)=(p_j\wedge 1-p_j)>0.$$
Further $Q_n=\sum_{j=1}^n q(f_j)=\sum_{j=1}^n(p_j\wedge 1-p_j)$, and  $b^2_n=   B_n = \sum_{j=1}^n p_j(1-p_j)$.  Thus 
$$\lim_{n\to \infty}b_n= \infty, \qq \limsup_{n\to \infty} \ {b_n^2}/{Q_n}=
 \limsup_{n\to \infty} \ \frac{\sum_{j=1}^n p_j(1-p_j)}{\sum_{j=1}^n(p_j\wedge 1-p_j)}<\infty.$$
As moreover the CLT holds,   we infer from  the above cited result,
$$ \lim_{n\to \infty} \sup_{k} \Big|\sqrt{B_n} {\mathbb P}\{S_n=k\} -\frac{1}{\sqrt{2\pi}}e^{-  \frac{(k-m_n)^2}{2B_n }}\Big|=0, $$
namely the LLT holds either.
 \end{proof}

\subsection{The characteristic function of $  S_n$.} \label{s2.2} 
Let $\p_k(t)$ be the characteristic function of $\xi_k$, $\p_k(t)= \E e^{2i\pi \xi_k t}$. Let also 
$$\Phi_n (t)=\E e^{2i\pi tS_n}= \prod_{k=1}^n\p_k(t) ,  $$
be the characteristic function of $S_n$.  We  prove the following estimate with explicit constants.
\begin{proposition}\label{Cramer.Phin} We have
$|\Phi_n(t)|\le \exp\big\{- 2B_n\sin^2\pi t  \big\} $.
Further
\begin{equation*}
  \Phi_n (t)= e^{ 2i \pi tm_n -  2   B_n  (\pi t)^2  + E_n(t)  },
    \qq \quad {\text with}\qquad |E_n(t)| \,\le\,  12\,m_n\, (\pi  |t|)^3.
\end{equation*}
In particular,  
\begin{equation*}
  \E e^{ iy\, \frac{S_n-m_n}{\sqrt {B_n}} }= e^{    - \frac{y^2}{2}    + E_n(y)  }, \qq \quad {\text and}\qquad |E_n(y)| \, \le  2 \, \frac{\sqrt {\log n}\,|y|^3}{\sqrt {n}}   .
 \end{equation*}

 \end{proposition} 
  For the proof, we use 
  the Lemma below which  is inspired by Lemma 3 in  Freiman-Pitman \cite{FP}. 
 The goal being to    obtain an   estimate of $\Phi(t)$  with
 explicit constants, this however requires to base the proof 
on   different  calculations. In particular we use
    the convenient estimate (Lemma 4.14 in Kallenberg \cite{Ka}), 
 \begin{equation}\label{expit}\Big|e^{ix} -\sum_{k=0}^n \frac{(ix)^k}{k!}\Big|\,\le\, \frac{2|x|^n}{n!}\wedge \frac{ |x|^{n+1}}{(n+1)!},
 \end{equation}
valid for any $x\in\R$ and $n\in \Z_+$.   
  \begin{lemma}\label{lfp1} Let $m  $ be a positive real and $p$ be a real such that $0<p<1$. Let $\b$ be a random variable defined by ${\mathbb P}\{ \b=0\}= p$, $
{\mathbb P}\{
\b=m\}= 1-p=q$.  Let $\p(t) ={\mathbb E\,} e^{2i\pi t \b}$. Then
 we have the following estimates,

\smallskip
\hspace*{1.2mm}{\rm (i)} For all real $t$,
 $|{\varphi}(t)|\le \exp\big\{- 2pq\sin^2\pi t m\big\} $

\smallskip
{\rm (ii)} If $q|\sin \pi t m|\le 1/3$, then
 \begin{eqnarray*}  \p(t)&=& e^{  q      2i\pi mt   - 2qp (  \pi mt)^2 +E}
 , 
\end{eqnarray*}
with  
$|E|\, \le \,   12 q (\pi m|t|)^3$.
    \end{lemma}

  \begin{proof} (i) One verifies that $|\p(t)|^2=1- 4 pq \sin^2\pi mt.$ As moreover $1-\t\le e^{-\t}$ if  ${\vartheta} \ge 0$, we obtain $|\p(t)|^2\le e^{-4 pq \sin^2\pi mt}$. 

\vskip 3 pt  (ii) Let  $|u|\le u_0<1$. From the series expansion of $\log(1+u)$,  it follows that
$$\log(1+u)=u-\frac{u^2}{2}+R, \qq \qquad|R|\le |u|^3\,\sum_{j=0}^\infty \frac{|\theta|^j}{3+j}\le \frac{|u|^3}{3(1- u_0 )}.$$
 Then $ 1+u= \exp\{ u -\frac{u^2}{2} +B\}$,  with $|B|\le C_0\,|u|^3$ 
and  $C_0=\frac{1}{3(1- u_0 )}$.

 Writing that  $\p(t)= 1 +
q\big( e^{2i\pi mt} -1\big)=1 + u$,
  where $|u|= 2q|\sin\pi mt|\le  2q\big(1\wedge   \pi m|t| \big) $, we obtain 
  \begin{equation}\label{phi.u} \p(t)= 1+u=e^{q ( e^{2i\pi mt} -1 ) }\, e^{  -\frac{u^2}{2} +B}, \qq 
\quad |B|\le 8C_0\,q^3 |\sin\pi mt|^3.
\end{equation}
   
   In order to   estimate $u^2$, we let  $A(t)= e^{2i\pi mt} -1-  2i\pi mt+ \frac{(2 \pi mt)^2 }{2}$, and    write  $ (  e^{2i\pi mt} -1 )^2$ under the form 
\begin{align*} 
   &\   A(2t)-2A( t) -\big\{  -1-  4i\pi mt+ 8 ( \pi mt)^2\big\} +2\big\{ -1-  2i\pi mt+ \frac{(2 \pi mt)^2 }{2}\big\}+1
\cr &=
   A(2t)-2A( t) -     ( 2\pi mt)^2.
\end{align*}
Then
$ ( e^{2i\pi mt} -1 )^2+     ( 2\pi mt)^2 =    A(2t)-2A( t)$,
 and so   
  $\frac{u^2}{2}=-  2q^2(  \pi mt)^2+  \frac{q^2}{2}(A(2t)-2A( t)).$ 
 
\vskip 3 pt
Let $u_0=\frac{2}{3}$ so that $C_0=1$. We assumed $q|\sin \pi t m|\le 1/3$, thus $|u|\le 2/3$. We consequently  get with \eqref{phi.u}, 
\begin{eqnarray}\label{phi.bound} \p(t)&=&  e^{q ( e^{2i\pi mt} -1 )  -\frac{u^2}{2} +B }\, =\, 
e^{q ( e^{2i\pi mt} -1 ) +    2q^2(  \pi mt)^2-  \frac{q^2}{2}(A(2t)-2A( t))+B}
\cr &=&e^{  q      2i\pi mt   - 2qp (  \pi mt)^2 +H+B}
, 
\end{eqnarray}
with $H=q A(t) -    \frac{q^2}{2}(A(2t)-2A( t))$. By estimate \eqref{expit}, letting $\d(x)=x^2\big(1\wedge   \frac{|x|}{6} \big)$,
\begin{eqnarray}\label{At.est}
 |A(t)|&\le&
 \d(2\pi m|t|)
.\end{eqnarray}
Using the rough bound  $\d(x)\le  \frac{|x|^3}{6} $, and   $|B|
\le\,8q^3\big(1  \wedge \pi m | t|\big)^3$, we get 
 \begin{equation}\label{HB.bound} |H|+|B|\, \le \, 2q|A(t)|+\frac{q^2}{2}|A(2t)| +|B|\, \le \, \big(\frac83 q+\frac43q^2+8q^3\big)(\pi m|t|)^3
\, \le \,  12 q (\pi m|t|)^3.
\end{equation}
 We conclude by inserting estimate  \eqref{HB.bound} into \eqref{phi.bound}.  \end{proof}

\begin{proof}[Proof of Proposition \ref{Cramer.Phin}]
Here we have $m=1$, $q=\frac{1}{\log k}$. As $|{\varphi}_k(t)|\le \exp\big\{- 2(1-\frac{1}{\log k})(\frac{1}{\log k})\sin^2\pi t  \big\}$, we get 
 $$|\Phi_n(t)|\le \exp\big\{- 2\sum_{k=3}^n(1-\frac{1}{\log k})(\frac{1}{\log k})\sin^2\pi t  \big\}.$$
Further condition $q|\sin \pi t m|\le 1/3$ in Lemma \ref{lfp1}  reduces for $\p_k(t)$ to 
$ \frac{1}{\log k}|\sin \pi t  |\le 1/3$.  Thus 
\begin{equation}
  \label{chf.xik}
  \p_k (t)= e^{ 2i \pi    ( \frac{1}{\log k}) t - 2\pi^2  (1-\frac{1}{\log k})(\frac{1}{\log k} ) \,t^2  + C_k(t)  },
\end{equation}
where
 $|C_k(t)|
\le\, \frac{12  (\pi  |t|)^3}{\log k}$. 
  Recalling that $m_n =\sum_{j=3}^n \frac1{\log j}$, $B_n= \sum_{j=3}^n \frac1{\log j}(1-\frac1{\log j})$, it follows that 
\begin{equation}
  \label{chf.sn}
\E e^{2i\pi tS_n}=  \Phi_n (t)= e^{ 2i \pi tm_n -  2\pi^2  B_n  t^2  + D_n(t)  },
\end{equation}
and
$|D_n(t)| \,\le\,  12\,m_n\, (\pi  |t|)^3  $. 
In particular,  
 \begin{equation}
  \label{chf.sn1}
 \E e^{ iy\, \frac{S_n-m_n}{\sqrt {B_n}} }= e^{    - \frac{y^2}{2}    + E_n(y)  },
 \end{equation}
with  $|E_n(y)| \,\le\,  \frac{3m_n}{2} \, (\frac{y}{\sqrt {B_n}})^3\,\le\,  2 \, \frac{\sqrt {\log n}|y|^3}{\sqrt {n}}  $. 
  This achieves the proof.
\end{proof}
 
\vskip 7 pt

One can derive from  Proposition \ref{Cramer.Phin}   the following value distribution result on  divisors of $S_n$. We omit the proof.  
 \begin{theorem}\label{div.Sn1}  Let $\Theta (d,m ,B)$ be the    elliptic Theta function defined by $$  \Theta (d,m ,B)  =  \sum_{\ell=0}^\infty \cos \big(2 m\pi\frac{\ell}{    d }\big)e^{- \frac{B\pi^2\ell^2}{  2 d^2}}. 
 $$

 We have 
\begin{eqnarray}\label{unif.elliptic.estim}   \sup_{2\le d\le n}\,\Big|\P\{d|S_n\}-\frac{\Theta (d,m_n,B_n)}{d}\Big|&=& 
\mathcal O \Big( \frac{( \log n)^3  }{     n} \Big).
\end{eqnarray}
\end{theorem}


\subsection{Remarks complementary  to   Cram\'er's proof.}\label{s2.3}

Recall   the way \eqref{C.Pngap} is established. We provide details, which are  necessary for the sequel.
 Let $c>0$, and set
$$E_m=\big\{\xi_{m+j}=0, 1\le j\le c (\log m)^2\big\}, \qq m\ge 2.$$
Then 
\begin{equation}\label{cramer.Em.est} \P\{E_m\}= \prod_{1\le j\le c (\log m)^2} \Big(1-\frac{1}{\log (m+j)}\Big)\asymp \frac{1}{m^c}.  
\end{equation}
Introduce the increasing sequence of integers 
$m_1=2$ and  $m_{r+1}=m_r + [c (\log m_r)^2]+1$. 
 \vskip 3 pt 
--- If $c\le 1$, the events $E_{m_r}$ being independent, by the second Borel-Cantelli lemma
$$\P\{\limsup_{r\to \infty}E_{m_r}\}=1,$$
in other words, almost surely, $S_{m_{r+1}}-S_{m_r}\ge c (\log m_r)^2$, $r$ infinitely often. 
\vskip 3 pt
--- If $c> 1$,  by  the first Borel-Cantelli lemma, the above probability is $0$. Thus with probability one, $S_{m +[c (\log m)^2]}-S_{m}\le c (\log m)^2$, $m$ ultimately. 
\vskip 3 pt
This suffices to imply \eqref{C.Pngap}, indeed:
\vskip 5 pt

(i)  If $c>1$, let  $P_{\nu(m)}$ denote the greatest $P_\nu$ less than $S_m$. As with probability one $\xi_{m+j}=1$, for some $1\le j\le c (\log m)^2$, $m$ ultimately, $P_{\nu(m+1)}$ cannot be larger than $S_{m +[c (\log m)^2]}$. Thus $P_{\nu(m+1)}-P_{\nu(m )}\le c (\log m)^2$, $m$ ultimately, almost surely. Now by Lemma \ref{l1cramer}, $S_{m }\sim \frac{m  }{\log m}$, almost surely, thus $\log S_{m }\sim  \log m$. We deduce that for any $c'>c$ fixed, $P_{\nu(m+1)}-P_{\nu(m )}\le c'\, (\log P_{\nu(m )})^2$, $m$ ultimately, almost surely, which naturally implies that the limsup in \eqref{C.Pngap} is less or equal to 1.  

 In addition, letting $m=m_r$, $r\ge 1$   in the above, we   have since $m_{r+1}=m_r + [c (\log m_r)^2]+1$, that with probability one: for any   $r$ large enough, there exists  a jump in the interval $[m_{r },m_{r+1}]$. Therefore 
\begin{equation}\label{Pr} \P\{ \exists r_0:\forall r\ge r_0,\quad P_r\le m_{r+1}+m_{r_1}\}= 1.
\end{equation}
\vskip 5 pt
(ii)  If $c\le 1$,  then almost surely $P_{\nu(m_{r+1})}-P_{\nu(m_r )}
 \ge c (\log m_r)^2\ge c (\log P_{\nu(m_r )})^2$, $r$ infinitely often. 

Thus  \eqref{C.Pngap} is supported by the   probability of the   set $\displaystyle{\limsup_{   r\to \infty}}\,E_m$ (resp. $\displaystyle{\limsup_{   r\to \infty}}\,E_{m_r}$) which is trivially 0 (resp. 1) depending on $c>1$ (resp. $c\le 1$). 

\begin{remark}\label{sna} As \eqref{C.Pngap} does not depend on the first random variables $\xi_j$, it follows that  the jumps   associated to  the truncated sequence $S^a_n=\sum_{j=a}^n \xi_j$, $a$ integer, $n>a$, also satisfy the same asymptotic property. Therefore there is no harm in Cram\'er's conjecture to consider instead the \lq primes\rq\, associated to this one.
\end{remark}
We now indicate several interesting   results complementing \eqref{C.Pngap}.  \vskip 2 pt

(1) 
Let   
$ N_J= \sum_{r=1}^J \chi_{E_{m_r}}$,  
$J\ge 1$, 
 and  put $B_J = \sum_{r=1}^J\P(E_{m_r})\bigl(1-\P(E_{m_r})\bigr)$. 
By
Kolmogorov's LIL,
\begin{equation}\label{cramer.LIL.Em} \P\Big\{ \limsup_{J\rightarrow \infty}\frac{N_J - \sum_{r=1}^J\P(E_{m_r})
}{  \sqrt{2B_J  \log \log
 B_J}} = 1
 \Big\}\,= 1.  
 \end{equation}

(2) Further, by Berry--Esseen inequality \cite{P},
\begin{equation}\label{cramer.BE.Em}
  \sup_{x\in \R} \biggl| \P\Biggl\{  \frac{N_J -\sum_{r=1}^J\P(E_{m_r})
}{   \sqrt{ B_J }}<x\Biggr\} -\frac{1}{  \sqrt 2\pi} \int_{-\infty}^x
e^{-u^2/2}du \biggr| = \, \mathcal O \bigg(\frac{1}{  \sqrt{
 \sum_{r=1}^J\P(E_{m_r}) }}\bigg).
\end{equation}
 
(3)  Let $0<c<1$. From \eqref{cramer.LIL.Em}, \eqref{cramer.Em.est} and   $m_r \asymp c\,r  (\log  r)^2$, it follows that with probability one
 \begin{equation}\label{cramer.Counting.Em}N_J 
  \asymp J^{1-c}(\log  J)^{-2c}, \qq \hbox{for any $J$ large enough}.
  \end{equation}
 Whence it follows that  for any $J$ large enough, the interval $[1, cJ  (\log  J])^2]$ contains at least $C\,J^{1-c}(\log  J)^{-2c}$ \lq primes\rq\, $P_\nu$ with large gaps, namely such that $ P_{\nu +1}-P_{\nu}\ge c\log^2 P_{\nu }$.
 Therefore the Cram\'er model  also predicts that    with probability one, for any $J$ large enough, the number of  large gaps between \lq primes\rq \, in the interval $[1, cJ  (\log  J])^2$ is $\O_c(J^{1-c}(\log  J)^{-2c})$.

\subsection{Value distribution of the divisors of the Bernoulli sum.}
 \label{appendix-3} The general problem of estimating the probability $\P\big\{d|B_n\big\}$   amounts  to the one of   estimating the $d$ series
\begin{equation}\label{vd.div.series}
\sum_{j }e^{-\frac{(jd-\rho)^2}{2n}},
\end{equation}
in which $d\ge 2$ and $0\le \rho <d$ are integers. This is clear from the   following uniform estimate  proved in   \cite{W1}, see also \cite{W4}, to which we refer for details, and from Poisson's summation formula.
\begin{theorem}\label{unif.est.theta}\begin{equation*} \sup_{2\le d\le n}\Big|\P\big\{d|B_n\big\}- \frac{\Theta(d,n)}{  d}  \Big|= {\mathcal O}\big((\log n)^{5/2}n^{-3/2}\big),  
\end{equation*}
where $\Theta (d,n) $ is the elliptic Theta function
\begin{equation*}\label{theta}\displaystyle{ \Theta (d,n)  =  \sum_{\ell\in \Z} e^{in\pi\frac{\ell}{    d }-\frac{n\pi^2\ell^2}{  2 d^2}}, 
   }
   \end{equation*}
\end{theorem}
By  applying    Poisson's summation formula: for   $x\in \R,\  0\le \d\le 1$, 
\begin{equation}\label{poisson}\sum_{\ell\in \Z} e^{-(\ell+\d)^2\pi x^{-1}}=x^{1/2} \sum_{\ell\in \Z}  e^{2i\pi \ell\d -\ell^2\pi x},
 \end{equation}  
we further get  \begin{equation}\label{poisson.comp.theta.llt}\frac{\Theta(d,n)}{ d}  =\sqrt{  \frac{ 2}{  \pi
n}}  \sum_{ 
z\equiv 0\, (d)} e^{-\frac{ (2z-n)^2}{  2 n}}.
 \end{equation}    
The series in \eqref{poisson.comp.theta.llt} is   of the   type given in \eqref{vd.div.series}.  An indication of the behavior of these series when $d$ and $n$ may simultaneously vary, is given by  the already sharp estimate, 
\begin{equation}\label{Theta.dmu} 
 \big|\frac{\Theta (d,n)}{  d}-\frac{1}{  d}\big| \ \le \ \begin{cases}   \frac{ C}{  d}
 e^{ - \frac{n \pi^2 }{  2d^2}} & \qq \text{   if}\quad d\le \sqrt n,\cr r
       \frac{C }{ \sqrt n}   & \qq \text{  if}\quad \sqrt n\le d\le n.
 \end{cases}
 \end{equation} 
 
  The following estimate is also proved  in  \cite{W1}.
\begin{theorem}\label{Bernoulli.quasi.p} There exist a positive real $c $ and  positive constants $C_0$, $\zeta_0$ such that for $k$ large
enough   we have,
\begin{equation}\label{P-Bn,a} 
\Big|\hskip1pt\P\big\{ P^-( B_k ) > \zeta \big\}-     \frac{ e^{-\gamma} }{   \log
\zeta }\,\Big|  \le  \frac{ C_0 }{    \log^2 \zeta } \qq \qquad (\,\zeta_0\le  \zeta\le k^{c/\log\log k}\,) .
\end{equation}
\end{theorem}
  As $\P\big\{   B_n  \ \hbox{is prime} \big\}\le \P\big\{ P^-( B_n ) > \zeta \big\}$, the following corollary is immediate.
\begin{corollary}\label{Bnprime.ubound}  There exists an absolute constant $C_1$, such that for all $n$ large enough,
\begin{equation*} 
\P\big\{   B_n  \ \hbox{is prime} \big\}\le C_1\,    \frac{   \log\log n }{   c  \log  n}.
\end{equation*}
The constant $c$ is the same as in  Theorem \ref{Bernoulli.quasi.p}.\end{corollary}
By    using Borel-Cantelli Lemma  it follows that, along subsequences   of integers growing at least   exponentially,  $B_n$ is  ultimately {\it not} prime   with probability 1.  
   \vskip 10 pt
 The  collection of  variables ${\mathfrak D} 
  =\big\{ \chi\{d|B_n\},    \  n\ge 1, d\ge 1\big\}
 $   
further  forms a mixing system, that is   the correlation function 
  \begin{equation}\label{Delta.def}  {{ \Delta}}  \big( (d,n),  (\d, m)\big)   \,=\, \P\big\{ d|B_n\, ,\,   \d|B_m\big\}-\P \{ d|B_n \}\P \{    \d|B_m \},
 \end{equation} 
verifies for each     $d$ and $\d$, 
\begin{equation}\label{corr1a1}
   \lim_{m-n\to \infty }  {{ \Delta}}  \big( (d,n),  (\d, m)\big)\,=\, 0. 
\end{equation}
The second order theory of this system is thoroughly studied in  \cite{W4}.  Three zones of dependence, weak independence and strong independence   can be identified,
   according to the cases   $n<m\le n+n^c$,   $n+ n^{c} \le m\le 2n$ and   $m\ge 2n$, where  $0<c<1$.
 The corresponding correlation estimates are established  in \cite{W4}.


\end{document}